\documentclass{amsart}
\usepackage[usenames]{color}

\usepackage{float}

\usepackage{soul}
\usepackage[normalem]{ulem}

\usepackage{pgfplots}
\usepgfplotslibrary{polar}

\usetikzlibrary{automata,positioning,arrows}
\usetikzlibrary{arrows.meta,calc,decorations.markings,math,arrows.meta}

\usepackage{amsmath,amssymb,amsthm,amsfonts,enumerate,url,mathbbol}
\usepackage{mathrsfs,tikz,graphicx,pifont, svg,stackrel}
\usepackage{hyperref}

\usepackage{aliascnt}

\usepackage{pgfplots}
\usepgfplotslibrary{polar}

\usetikzlibrary{automata,positioning}
	
\usepackage[utf8]{inputenc}

\usepackage{verbatim}
\usepackage{array}

\def\be{\begin{equation}}
\def\ee{\end{equation}}
\def\beq{\begin{eqnarray}}
\def\eeq{\end{eqnarray}}
\def\beqs{\begin{eqnarray*}}
\def\eeqs{\end{eqnarray*}}
\def\ea{\end{array}}
\def\ea{\end{array}}
\def\bt{\begin{theo}}
\def\et{\end{theo}}
\def\bl{\begin{lemma}}
\def\el{\end{lemma}}
\def\br{\begin{rems}}
\def\er{\end{rems}}
\def\bc{\begin{coro}}
\def\ec{\end{coro}}

\theoremstyle{plain}
\newtheorem{theo}{Theorem}[section]

\newaliascnt{cor}{theo}
\newaliascnt{prop}{theo}
\newaliascnt{lemma}{theo}

\newtheorem{lemma}[lemma]{Lemma}
\newtheorem{prop}[prop]{Proposition}
\newtheorem{cor}[cor]{Corollary}

\aliascntresetthe{cor}
\aliascntresetthe{prop}
\aliascntresetthe{lemma}

\theoremstyle{definition} 

\newaliascnt{defi}{theo}
\newaliascnt{assum}{theo}
\newaliascnt{assums}{theo}
\newaliascnt{prob}{theo}

\newtheorem{assums}[assums]{Assumptions}

\aliascntresetthe{defi}
\aliascntresetthe{assum}
\aliascntresetthe{assums}
\aliascntresetthe{prob}

\theoremstyle{remark}

\newaliascnt{rems}{theo}
\newaliascnt{rem}{theo}
\newaliascnt{exa}{theo}
\newaliascnt{exs}{theo}

\newtheorem{rems}[rems]{Remarks}
\newtheorem{rem}[rem]{Remark}

\aliascntresetthe{rems}
\aliascntresetthe{rem}
\aliascntresetthe{exa}
\aliascntresetthe{exs}

\numberwithin{theo}{section}
\numberwithin{equation}{section}
\numberwithin{figure}{section}

\DeclareMathOperator{\Id}{Id}
\DeclareMathOperator{\Ker}{Ker}
\DeclareMathOperator{\Ran}{Ran}

\DeclareMathOperator{\diag}{diag}
\DeclareMathOperator{\sign}{sgn}

\DeclareMathOperator{\lin}{span}

\DeclareMathOperator{\rank}{rank}

 \def\mG{\mathsf{G}} 
 
 \def\mV{\mathsf{V}}
 
 \def\mE{\mathsf{E}}

 \def\K{\mathbb{K}}

 \def\mmu{\mathsf{u}} 
 \def\mmv{\mathsf{v}}
 \def\mmx{\mathsf{x}} 
 \def\mmy{\mathsf{y}} 
  \def\mmh{\mathsf{h}} 
 \def\mv{\mathsf{v}}
 \def\me{\mathsf{e}}
 \def\mw{\mathsf{w}}
 
 \def\mf{\mathsf{f}}
 \def\mg{\mathsf{g}}
 \def\my{\mathsf{y}}

\def\cL{\mathcal{L}}

\newcommand{\R}{\mathbb{R}}
\newcommand{\C}{\mathbb{C}}
\newcommand{\N}{\mathbb{N}}

\newcommand{\bL}{\mathbf{L}}

\title{Dynamic Transmission Conditions for Linear Hyperbolic Systems on Networks} 

\author[M.~Kramar Fijav\v{z}]{Marjeta Kramar Fijav\v{z}}
\address{Marjeta Kramar Fijav\v{z}, University of Ljubljana, Faculty of Civil and Geodetic Engineering, Jamova 2, SI-1000 Ljubljana, Slovenia / Institute of Mathematics, Physics, and Mechanics, Jadranska 19, SI-1000 Ljubljana, Slovenia}
\email{marjeta.kramar@fgg.uni-lj.si}

\author[D.~Mugnolo]{Delio Mugnolo}
\address{Delio Mugnolo, Lehrgebiet Analysis, Fakult\"at Mathematik und Informatik, Fern\-Universit\"at in Hagen, D-58084 Hagen, Germany}
\email{delio.mugnolo@fernuni-hagen.de}

\author[S.~Nicaise]{Serge Nicaise}
\address{Serge Nicaise, Universit\'e Polytechnique Hauts-de-France, LAMAV,
FR CNRS 2956, 
F-59313 - Valenciennes Cedex 9 France}
\email{Serge.Nicaise@uphf.fr}



\thanks{The work of M.K.F.\ was partially supported by the Slovenian Research Agency, Grant No.~P1-0222, and the work of D.M.~by the Deutsche Forschungsgemeinschaft (Grant 397230547). All the authors would like to acknowledge networking support by the COST Action CA18232. }

\begin{document}

\begin{abstract}
We study evolution equations on networks that can be modeled by means of hyperbolic systems. We extend our previous findings in~\cite{KraMugNic20} by discussing well-posedness under rather general transmission conditions that might be either of stationary or dynamic type - or a combination of both. Our results rely upon semigroup theory and elementary linear algebra. We also discuss qualitative properties of solutions.
\end{abstract}

\maketitle

\section{Introduction}

The present paper represents the second part of our investigations on linear hyperbolic systems.
Given a metric graph $\mathcal G$, i.e., a graph $\mG=(\mV,\mE)$ each of whose edges $\me\in\mE$ is identified with an interval $(0,\ell_\me)\subset \R$,
we are going to study evolution equations of the form 
\begin{equation}\label{eq:max1}
\dot{u_\me }(t,x)=M_\me(x)u'_\me(t,x)+ N_\me(x) u_\me (t,x), \quad t\ge 0,\ x\in (0,\ell_\me),\ \me\in \mE,
\end{equation}
where $u_\me$ is a vector-valued function of size $k_\me\in \N_1:=\{1,2,3, \ldots\}$, and $M_\me$ and $N_\me$ are 
 matrix-valued functions of size $k_\me\times k_\me$. 
Hence, each of these equations is supported on an edge: they are going to be coupled by means of suitable transmission conditions in the vertices. In~\cite{KraMugNic20} we have proposed a parametrization of such conditions that bears some similarity to the boundary conditions for scalar-valued, multi-dimensional transport equations studied in~\cite{LaxPhi60}. The goal of this paper is to extend it to general conditions that may be either of stationary, like in~\cite{KraMugNic20}; or of dynamic type.

In the case of systems of parabolic equations, dynamic boundary conditions have been studied at least since~\cite{Nic85} and classically interpreted as conditions of Wentzell-type arising in the theory of stochastic processes, see~\cite{Mug10} and references therein. For hyperbolic systems, however, dynamic boundary condition have been discussed far less frequently in the literature; specific classes of problems arising in applied mathematics have been investigated in~\cite{CurClaZha04,FerMisQua05,RuaClaZha07,BecImpJol14,Bec16}(system of first order problems) and~\cite{MorgulRaoConrad:94,HansenZuazua:95,Castro:97,ConradMorgul:98,GreMug20b} (systems of strings and/or beams with point masses at the junctions). 

In this paper we propose a unified formalism to capture hyperbolic systems with hybrid transmission conditions, including the extreme cases of purely stationary or  purely dynamic conditions; in fact, we can also allow for conditions that are dynamic only at some vertices and on some of the unknown's components; as we will see, this rather general setting is motivated by applications and leads to introducing a block operator matrix
\[
{\mathbb A}:=\begin{pmatrix}
\mathcal A & 0\\
\mathcal{B} & { \mathcal C}
\end{pmatrix}
\]
with coupled (i.e., non diagonal) domain on a suitable direct sum of Hilbert spaces: $\mathcal A$ is a first order differential operator encoding the dynamics driving~\eqref{eq:max1}, while operators $\mathcal{B,C}$ model (possibly nonlocal) damping phenomena in the vertices.

Just like in~\cite{KraMugNic20}, our main assumptions involve the existence of a \textit{Friedrichs symmetrizer}, an idea that goes back to~\cite{FriLax65}. We are going to show that unlike in the canonical setting considered in the literature, however, the Friedrichs symmetrizer of a hyperbolic system with dynamic boundary condition is an operator \textit{matrix}, with additional terms that control the boundary space -- a subspace of functions supported on a graph's vertices, in the case most relevant for us.

It is known from the theory of parabolic and wave equations with dynamic boundary conditions that boundary operators of higher order are useful to model close feedbacks that may stabilize the system. The role of the operator that couples the hyperbolic evolution with the boundary dynamics -- $\mathcal B$, in the notation above -- is even more central in the present context: indeed, we show that the dimensions of its range and null space directly impacts on the maximality of $\mathbb A$, and hence on the well-posedness of the associated Cauchy problem, see our main \autoref{prop:main1dyn}; backward well-posedness as well as energy conservation or decay properties can be characterized in terms of boundary conditions, too. These results, presented in \autoref{sec:dynam}, contain our main findings from~\cite{KraMugNic20} as special cases; they can be regarded as a parametrization of infinitely many realizations enjoying particularly good properties.

In \autoref{sec:qualitative} we then discuss qualitative properties enjoyed by solutions of our hyperbolic systems: in particular, we consider two relevant order intervals of the Hilbert space and discuss their invariance under the semigroup that governs the system by presenting sufficient (and, sometimes, necessary) conditions on the boundary conditions. 

In \autoref{sec:examples} we revisit some known hyperbolic-type equations with dynamic conditions, including transport equations~\cite{Sik05}, a second sound model~\cite{Rac02}, and a 1D Maxwell system~\cite{Bec16}. We also consider Dirac equations on networks, for which a parametrization of infinitely many realizations governed by a unitary group (resp., contractive semigroup) was first studied in~\cite{BolHar03} (resp., \cite{KraMugNic20}); we show that infinitely many further relevant realizations naturally arise by allowing for dynamic conditions.

We furthermore study qualitative properties of solutions of these equations by  applying our abstract theory. It turns out that the above mentioned conditions for invariance are rather restrictive: while real-valued initial data give rise to real-valued solutions in most applications, we see that positivity or a priori estimates in $\infty$-norm for the solutions can be seldom observed.

\section{General setting}\label{sec:general}

We are going to collect different sets of assumptions that we are going to impose in the following; roughly speaking, they are of combinatorial, analytic, and operator theoretical nature, respectively.

\begin{assums}\label{assum-graph} 
$\mG=(\mV,\mE)$ is a nonempty, finite \textit{combinatorial graph}, $(k_\me)_{\me\in \mE}$ is a family of positive integers and $(\ell_\me)_{\me\in \mE}$ is a family of positive numbers.
\end{assums}

In the following, we adopt the notation 
\[
k:=\sum_{\me\in\mE} k_\me\quad\text{and} \quad k_\mv:=\sum_{\me\in\mE_\mv} k_\me,
\]
where $\mE_\mv$ is the set of all edges incident in $\mv$. 
Notice that
\begin{equation}\label{eq:2k}
\sum_{\mv\in\mV}k_\mv=2k,
\end{equation}
by the Handshaking Lemma.

We rakishly turn $\mG$ into a \textit{metric graph} (or \textit{network}) $\mathcal G$ by identifying each $\me\in \mE$ with an interval $[0,\ell_\me]\subset\R$; a more precise definition can be found in~\cite{Mug19}. We further impose standard assumptions on the coefficient matrices $M,N$ that appear in~\eqref{eq:max1}; additionally, we require the existence of a Friedrichs symmetrizer $Q$. 
\footnote{Here and in the following we denote by $M_{n,m}(\K)$ the space of all $n\times m$-matrices on the field $\K$; and $M_n(\K):=M_{n,n}(\K)$.}

\begin{assums}\label{assum-MQN} 
For each $\me\in\mE$, $M_\me,N_\me:[0,\ell_\me]\to M_{k_\me}(\C)$ are mappings such that the following hold.
\begin{enumerate}[(1)] 
\item $[0,\ell_\me]\ni x\mapsto M_\me(x)\in M_{k_\me}(\C)$ is 
Lipschitz continuous; and $M_\me(x)$ is invertible for each $x\in [0,\ell_\me]$.
\item $[0,\ell_\me]\ni x\mapsto N_\me(x)\in M_{k_\me}(\C)$ is of class $L^\infty$.
\item\label{assum:Q} There exists a Lipschitz continuous mapping
$[0,\ell_\me]\ni x\mapsto Q_\me(x)\in M_{k_\me}(\C)$ such that 
\begin{enumerate}[(i)]
\item\label{assumMe} $Q_\me(x)$ and $Q_\me(x) M_\me(x)$ are Hermitian for all $x\in [0,\ell_\me]$ and
\item
\label{sn:29/1:1} $Q_\me(\cdot)$ is \emph{uniformly positive definite}, i.e.,
there exists $q>0$ such that 
\[Q_\me(x)\xi \cdot \bar \xi \geq q \|\xi \|^2
\text{ for all }\xi \in \C^{k_\me} \text{ and } x\in [0,\ell_\me].\]
\end{enumerate}
\end{enumerate}
\end{assums}
\autoref{assum-MQN} are identical with~\cite[Assumptions~2.1]{KraMugNic20}.

We introduce for each $\mv\in\mV$
 the trace operator $\gamma_\mv: \bigoplus_{\me\in\mE}
 H^1(0,\ell_\me)^{k_\me}\to \C^{k_\mv}$ defined by \[
\gamma_\mv(u):= 
\left(u_\me(\mv)\right)_{\me\in\mE_\mv},\qquad \mv\in\mV,
\]
and the $k_\mv\times k_\mv$ block-diagonal matrix $T_\mv$ with $k_\me\times k_\me$ diagonal blocks 
\begin{equation}
\label{eq:defTv}
T_\mv:= 
\diag \left( Q_\me(\mv) M_\me(\mv) {\iota}_{\mv \me}
\right)_{\me\in\mE_\mv},\qquad \mv\in\mV,
\end{equation}
where we recall that the $|\mV|\times |\mE|$ (signed) \textit{incidence matrix} $\mathcal I=(\iota_{\mv\me})$ of the graph $\mG$ is defined by
\begin{equation}
\label{eq:defIv}
 \mathcal I:=\mathcal I^+-\mathcal I^-
\end{equation}
with $\mathcal I^+=(\iota^+_{\mv\me})$ and $\mathcal I^-=(\iota_{\mv\me}^-)$ given by
\[
{\iota}_{\mv \me}^+:=\left\{
\begin{array}{ll}
1 & \hbox{if } \mv \hbox{ is terminal endpoint of } \me, \\
0 & \hbox{otherwise,}
\end{array}\right.
\qquad
{\iota}_{\mv \me}^-:=\left\{
\begin{array}{ll}
1 & \hbox{if } \mv \hbox{ is initial endpoint of } \me, \\
0 & \hbox{otherwise.}
\end{array}\right.
\]

Unlike in our earlier work~\cite{KraMugNic20}, our aim is to develop a setting that will eventually allow us to impose dynamic boundary conditions on a subset of the vertex set $\mV$. Ideas that go back to~\cite{AmaEsc96,FavGolGol01,AreMetPal03} suggest to study the relevant evolution equation as a Cauchy problem on a larger Hilbert space. The necessary formalism can be introduced as follows.

 \begin{assums}\label{assum-spaces}
For each $\mv\in \mV$ the following holds.
\begin{enumerate}[(1)]
\item $Y^{(d)}_\mv \subset Y_\mv$ are subspaces of $\mathbb{C}^{k_\mv}$;
\item $B_\mv\colon Y_\mv\to Y^{(d)}_\mv$ is a linear operator;
\item $C_\mv$ is a  linear operator on $Y^{(d)}_\mv$;
\item $Q_\mv$ is a hermitian and positive definite operator on ${Y}^{(d)}_\mv$.
\end{enumerate}
\end{assums}
We stress that the assumptions on $Q_\me$ and $Q_\mv$ are structurally different. While, given a system of differential equations, we can only study it by the means of the theory presented in this paper if we are able to \textit{find} suitable Friedrich symmetrizers $Q_\me$ leading to a Hermitian product $Q_\me M_\me$, in the following we are free to take $Q_\mv$ as we wish. The ``lazy'' choice of $Q_\mv=\mathbb I$ is always allowed, but the main results in \autoref{sec:dynam} show that  it pays off to pick $Q_\mv$ tailored to enforce energy conservation or decay. 

With these objects, we 
set 
\[
\bL^2(\mathcal G):=\bigoplus_{\me\in\mE} L^2(0,\ell_\me)^{k_\me}\qquad\hbox{and}\qquad Y^{(d)}:=\bigoplus\limits_{\mv\in \mV} Y^{(d)}_\mv
\]
and introduce the Hilbert space
\[
{\bL}^2_{d}(\mathcal G):=\bL^2(\mathcal G)\oplus Y^{(d)},
\]
equipped with the inner product
\begin{equation}\label{prod-d}
\begin{split}
\left(\begin{pmatrix} u \\ \mmx \end{pmatrix}, \begin{pmatrix} v \\ \mmy \end{pmatrix}\right)_d
&:= \sum_{\me\in\mE}\int_0^{\ell_\me} Q_\me(x)u_\me(x)\cdot\overline{v}_\me(x)\ dx+\sum_{\mv\in \mV} Q_\mv \mmx_\mv\cdot \bar \mmy_\mv,
\qquad u,v\in \bL^2(\mathcal G),\ \mmx,\mmy\in{Y^{(d)}},
\end{split}
\end{equation}
which is equivalent to the canonical one.
This is the function space setup we are going to use to deal with dynamic boundary conditions.

We stress that we are not assuming $B_\mv$ to be surjective, hence 
$\Ran B_\mv$ does not need to agree with $Y^{(d)}_\mv$. Accordingly, we split up
$Y^{(d)}_\mv$ as 
\begin{equation}\label{eq:Yd-split}
Y^{(d)}_\mv=\Ran B_\mv\oplus \Ker B_\mv^\ast,
\end{equation}
where the sum is orthogonal with respect to the   inner product of
$Y^{(d)}_\mv$ induced by the Euclidean inner product of $\C^{k_\mv}$.
We shall denote by $P_\mv^{(d)}$ (resp., $P_\mv^{(d,0)}$) the orthogonal projector of $\C^{k_\mv}$ onto $Y^{(d)}_\mv$ (resp.  of $Y^{(d)}_\mv$ onto 
 $\Ker B_\mv^\ast$), of course with respect to said inner product. In the same spirit, if
 $U$ is a vector space included into $\C^{k_\mv}$ (resp. $Y_\mv$), we denote by $U^{\perp}$ (resp. $U^{\perp_y}$) its orthogonal complement in $\C^{k_\mv}$ (resp.  $Y_\mv$) with respect to said inner product.

\section{Well-posedness of systems with dynamic vertex conditions}\label{sec:dynam}

Inspired by the discussion in \cite[\S~8.2]{Bec16}, where time-dependent transmission conditions for the 1D Maxwell's equation are derived by methods of asymptotic analysis,
we are going to introduce an abstract framework in order to investigate well-posedness of~\eqref{eq:max1} under general transmission conditions of dynamic type.

We first introduce 
 the linear and continuous operators ${\mathcal A}$ and 
$\mathcal{B}$ from 
\[
D_{\max}:= \bigoplus_{e\in \mE} H^1(0,\ell_\me)^{k_\me}
\]
to $\bL^2(\mathcal G)$ and $Y^{(d)}_\mv$, respectively,
by
\[
\begin{split}
({\mathcal A} u)_\me&:= M_\me u'_\me+ N_\me u_\me , \quad \me \in \mE,
\\
(\mathcal{B} u)_\mv&:= B_\mv \gamma_\mv(u), \quad \mmv\in \mV,
\end{split}
\]
as well as the operator $\mathcal C$ on $Y^{(d)}$ 
defined by
\[
(\mathcal C\mmx)_\mv :=C_\mv \mmx_\mv,\quad \mv\in\mV,
\]
and study the operator
\begin{eqnarray}\label{eq:opA}
{\mathbb A}:=\begin{pmatrix}
\mathcal A & 0\\
\mathcal{B} & { \mathcal C}
\end{pmatrix},
\end{eqnarray}
with domain
\begin{eqnarray}\label{eq:domA}
 D(\mathbb A):=\left\{\begin{pmatrix} u \\ \mmx \end{pmatrix}\in D_{\max}\oplus Y^{(d)}: 
\gamma_\mv( u )\in Y_\mv\hbox{ and }\mmx_\mv=P^{(d)}_\mv \gamma_\mv( u ) \hbox{ for all } \mv\in \mV \right\}.
\end{eqnarray}
The present setting is a strict generalization of the context discussed in our previous investigation~\cite{KraMugNic20}, where 
for all $\mv\in\mV$ we take $Y^{(d)}_\mv=\Ker B_\mv^\ast  = \{0\}$, $\Ker B_\mv=Y_\mv$.
In our main well-posedness results there -- \cite[Thm.~3.7 and Thm.~4.1]{KraMugNic20} -- we had to assume each $Y_\mv$ to be a subspace of the null or nonpositive isotropic cone of the quadratic form
\begin{equation}\label{eq:qv-def} 
q_\mv(\xi ):=T_\mv \xi \cdot \bar \xi,\qquad \xi\in \C^{k_\mv},
\end{equation}
i.e., $q_\mv(\xi)$ to be identically zero or nonpositive for all $\xi\in Y_\mv$ and all $\mv\in \mV$ (see~\cite[App.~C]{KraMugNic20} for more details), 
in order to control the boundary terms that arise from integration by parts when checking dissipativity of the relevant operator $\mathcal A$. In the present context, these conditions  have to be adapted. 
More precisely, the definition of $\mathbb A$ and computations analogous to those at the beginning of~\cite[\S3]{KraMugNic20}
show that for any ${\mathbb u} :={u \choose \mmx} \in D(\mathbb A)$,
\beq\label{eq:Adissdyn}
\begin{split}
\Re\left(\mathbb A {\mathbb u}, {\mathbb u}\right)_d
&=
 \Re\sum_{\me\in\mE}\int_0^{\ell_\me}
\left (Q_\me N_\me u_\me \cdot \bar u_\me \right)\, dx
-\frac{1}{2}\sum_{\me\in\mE}\int_0^{\ell_\me}
\left(Q_\me M_\me \right)' u_\me \cdot 
 \bar u_\me \,dx 
 \\
&\quad
+\frac{1}{2} \sum_{\mv\in\mV} T_\mv \gamma_\mv(u)\cdot\gamma_\mv(\bar u)
\\
&\quad
+ \Re \sum_{\mv\in\mV} \left(
Q_\mv \left(B_\mv +C_\mv P^{(d)}_\mv\right)\gamma_\mv(u) \cdot P^{(d)}_\mv \gamma_\mv(\bar u)
\right).
\end{split}
\eeq
Rearranging the terms and using the fact that
\begin{equation}\label{eq:QvYd}
 Q_\mv B_\mv \gamma_\mv(u)\cdot P^{(d)}_\mv\gamma_\mv(\bar u) = P^{(d)}_\mv Q_\mv B_\mv \gamma_\mv(u)\cdot \gamma_\mv(\bar u)  =Q_\mv B_\mv \gamma_\mv(u)\cdot\gamma_\mv(\bar u), 
 \end{equation}
 since $Q_\mv$ maps to $Y^{(d)}$,
we obtain
\beq\label{eq:Adissdyn-2}
\begin{split}
\Re\left(\mathbb A {\mathbb u}, {\mathbb u}\right)_d
&=
\frac{1}{2}\sum_{\me\in\mE}\int_0^{\ell_\me}
\left(Q_\me N_\me+N_\me^\ast Q_\me-(Q_\me M_\me \right)') u_\me \cdot 
 \bar u_\me \,dx 
 \\ 
 &\quad + \frac12 \sum_{\mv\in\mV} (Q_\mv C_\mv+ C^\ast_\mv Q_\mv )\mmx_\mv \cdot \bar \mmx_\mv 
+\frac{1}{2} \sum_{\mv\in\mV} (T_\mv +Q_\mv B_\mv+B_\mv^\ast Q_\mv) \gamma_\mv(u)\cdot\gamma_\mv(\bar u).
\end{split}
\eeq
We hence have two boundary terms: in $Y^{(d)}_\mv$ and in  the whole $Y_\mv$, respectively.

As in\cite[\S3]{KraMugNic20}, the maximality property of $\pm \mathbb A$ relies on a basis property of some specific vectors of $\mathbb{C}^k$. 
We first need to introduce some notations: we 
write
$I_\mv:=\{1,2,\ldots, {\dim Y_\mv^{\perp}}\}$, $J_\mv^{(R)}:=\{1,2,\ldots, \dim \Ran B_\mv\}$, $J_\mv^{(K)}:=\{1,2,\ldots, \dim \Ker B_\mv^\ast\}$ and fix bases $\{\mw^{(\mv, i)}\}_{i\in I_\mv}$, $\{\my^{(\mv, j)}\}_{j\in J_\mv^{(R)}}$, $\{ \mw_{KB^\ast}^{(\mv, l)}\}_{l\in J_\mv^{(K)}}$ of the subspaces
 $Y_\mv^{\perp}$, $\Ran B_\mv$, and $\Ker B_\mv^\ast$, respectively. Furthermore, let 
$
\mw_{RB^\ast}^{(\mv, j)}:=B_\mv^{\ast}  \my^{(\mv, j)}, j\in J_\mv^{(R)}.
$
Note that
\begin{equation}\label{eq:ranB*}
\lin \{{\mw_{RB^\ast}^{(\mv, j)}}\}_{{j\in J_\mv^{(R)}}}  =\Ran B_\mv^\ast  \subset Y_\mv
\end{equation}
and $\dim \Ran B_\mv = \dim \Ran B_\mv^\ast$.
Finally, we introduce the space
\begin{equation} \label{defZmv}
Z_\mv:={Y_\mv^{\perp} \oplus\left(  \Ran B_\mv^\ast  + \Ker B_\mv^\ast\right) }\subset \C^{k_\mv}
\end{equation}
which is spanned by the set of vectors
\begin{equation} \label{Zmv-span}
\mathcal{W}_\mv:=\{\mw^{(\mv, i)} : i\in I_\mv\} \cup\{ {\mw_{RB^\ast}^{(\mv, j)}} : j\in  J_\mv^{(R)}\}\cup\{\mw_{KB^\ast}^{(\mv, l)} : l\in J_\mv^{(K)} \}.
\end{equation}
 The choice of this space is guided by the proof of the maximality of the operator $\mathbb A$, see the proof of \autoref{prop:main1dyn}.

Any element $\mw\in \C^{k_\mv}$ can be identified with a vector 
$(\mw_\me)_{\me\in \mE_\mv}$ and we denote by 
$\widetilde \mw\in\C^k$ its extension to the whole set of edges, namely,
\be\label{eq:wtilde}
\widetilde \mw_\me :=
\left\{
\begin{array}{ll}
\mw_\me, &\hbox{ if } \me\in \mE_\mv,\\
0, &\hbox{ else. }
\end{array}
\right.
\ee
In the same way each coordinate of an element of {a subset} $U\subset\C^{k_\mv}$ corresponds to some $\me\in \mE_{\mv}$ and, as above, we can extend these sets to $\C^k$ by setting a $0$ in each coordinate corresponding to $\me$ whenever $\me\notin \mE_\mv$. We denote these extensions by $\widetilde{U}\subset\C^k$. 
Using this notation we will assume that 
\begin{equation}\label{eq:basis}
\text{the set  }\widetilde{\mathcal{W}}:=\bigcup_{\mv\in\mV} \widetilde{\mathcal{W}}_\mv \text{ is a basis of }\C^{k}.
\end{equation}

\begin{rem}\label{rem:B-surj}
Let us mention two special cases  when condition \eqref{eq:basis} can be reformulated in terms of dimension  equation. \\
(1)
First, note that in the case of only stationary boundary conditions -- i.e., when $Y^{(d)}_\mv=\{0\}$ and hence $\Ran B_\mv = \Ran B_\mv ^\ast = \Ker B_\mv^\ast = \{0\}$ and $\Ker B_\mv=Y_\mv$ -- we have $J_\mv^{(R)}=J^{(K)}_\mv=\emptyset$ and $Z_\mv=Y_\mv^{\perp}$.
By  \cite[Lemma~3.5]{KraMugNic20},  the set $\widetilde{\mathcal{W}}=\{\widetilde \mw^{(\mv, i)}\}_{i\in I_\mv, \mv\in \mV}$ is a basis of $\C^{k}$
 if and only if
{\[
 \dim \sum_{\mv\in \mV} \widetilde {Y_{\mv}^\perp}= k = \sum_{\mv\in \mV}  \dim Y_\mv.
\]}
\\
(2) Let us now more generally  consider the case of dynamic boundary conditions with surjective operator $B_\mv$. Then $Z_\mv$ reduces to the direct sum
\begin{equation} \label{defZmvBsurjective}
Z_\mv:=  Y_\mv^{\perp} \oplus \Ran B_\mv^\ast
\end{equation}
and $J^{(K)}_\mv=\emptyset$. In this case, $\widetilde{\mathcal{W}}_\mv$ is a basis of $ \widetilde Z_\mv$ and, by the same reasoning as in the proof of \cite[Lemma~3.5]{KraMugNic20} we see that \eqref{eq:basis} holds if and only if 
{\[
 \dim \sum_{\mv\in \mV} \widetilde {Z}_{\mv} = k= \sum_{\mv\in \mV}  \dim Z^\perp_\mv.
\]}
Observe that $Z^\perp_\mv = Y_\mv \cap (\Ran B^\ast_\mv)^\perp = \Ker B_\mv$. By the surjectivity of $B_\mv$ we further have $\dim\Ker B_\mv = \dim Y_\mv - \dim Y_\mv^{(d)}$ and thus \eqref{eq:basis} is equivalent to 
\begin{equation}\label{eq:basis-surj}
{
 \dim \sum_{\mv\in \mV} \widetilde {Z}_{\mv} = k= \sum_{\mv\in \mV}  \left(  \dim Y_\mv - \dim Y_\mv^{(d)} \right).
}\end{equation}
\end{rem}

\begin{rem}\label{rem:comp-par-hyp}
 Let us reverse our perspective and assume that we are interested in deriving new well-posed systems from known ones, rather than modelling problems with dynamic conditions stemming from applications; this is similar to the goal of extension theory in mathematical physics, where one is interested of describing as many realizations of a given Hamiltonian as possible, subject to the condition that such realizations are still governing a well-behaved PDE.
The condition in~\eqref{eq:basis-surj} shows that, in spite of superficial similarities, the present situation is different from that discussed in~\cite{GreMug20b} in the context of parabolic equations. Roughly speaking, the findings in~\cite{GreMug20b} show that, as soon a choice of a family of spaces $Y_\mv$, $\mv\in\mV$, define boundary conditions leading to well-posedness, each choice of subspaces $Y^{(d)}_\mv$ of $Y_\mv$, $\mv\in\mV$, will lead to a new well-posed system. As a matter of fact, modifying a well-posed hyperbolic system in order to allow for dynamic vertex conditions is a delicate issue: we will see in \autoref{sec:examples} that, starting from any well-posed hyperbolic system (say, taken from~\cite[\S~5]{KraMugNic20}) driven by the operator $\mathcal A$
with stationary conditions
\[
\gamma_\mv(u)\in Y^{(0)}_\mv
\]
 encoded in a space $Y^{(0)}_\mv$, switching to a dynamic setting requires to carefully enlarge  these spaces to find suitable $Y_\mv$ and at the same time allow for non-trivial $Y^{(d)}_\mv$, if we want \eqref{eq:basis-surj} to be satisfied. 
\end{rem} 

Next results extend \cite[Thm~3.7 and Thm.~4.1]{KraMugNic20} to the case where both dynamic and stationary conditions are allowed. We adopt the terminology of~\cite[Appendix~C]{KraMugNic20}.
Extending the statement to the case of $\lambda\ne 0$ might look superfluous, but it will prove useful when discussing concrete systems of PDEs, cf.\ Section~\ref{sec:secsound}.

\begin{theo}\label{prop:main1dyn}
For all $\mv\in \mV$, let~\eqref{eq:basis} hold and let moreover  $Y_\mv$ be a subspace of the nonpositive isotropic cone of the quadratic form associated with $T_\mv +Q_\mv B_\mv+B_\mv^\ast Q_\mv -\lambda P_\mv^{(d)} Q_\mv  P_\mv^{(d)}$ for some $\lambda\geq 0$.
Then $\mathbb A$ generates a strongly continuous semigroup on $\bL^2_d(\mathcal G)$. 
\end{theo}

\begin{proof} First of all, let us observe that $\mathbb A$ is densely defined by~\cite[Lemma~5.6]{MugRom07}.
As the operator $(u,\mmx)^\top \mapsto 
(Nu , C\mmx + P_\mv^{(d,0)})^\top$ is a bounded perturbation of $\mathbb A$, the claim will follow if we can prove that the operator matrix
\[
\mathbb A_0:=\begin{pmatrix}
M\frac{d}{dx} & 0\\ 
\mathcal{B} & -\mathcal{P}^{(d,0)}
\end{pmatrix},\qquad D(\mathbb A_0):=D(\mathbb A),
\]
with $( {\mathcal P}^{(d,0)}\mmx)_\mv = P^{(d,0)}_\mv \mmx_\mv$, 
that corresponds to $\mathbb A$ with the choice $N=0$ and  ${\mathcal{C}=- \mathcal{P}^{(d,0)}}$,
is $m$-quasidissipative.

Formula~\eqref{eq:Adissdyn-2} 
and the assumptions on matrices $Q_\me$ and $M_\me$
show that dissipativity
holds for $\mathbb A_0-\lambda \mathbb{I}$ on $D(\mathbb A)$; let us check maximality. 

{To this aim, for any ${f}\in \bL^2(\mathcal G)$ and any $\mg\in Y^{(d)}$, we first look for a solution ${\mathbb u}:=(u,\mmx)^\top \in D(\mathbb A)$ of
\[
\mathbb A_0 (u,\mmx)^\top= ({f},\mg)^\top,
\]
namely solution of
\[
M_\me(x)u'_\me(x)={f}_\me(x)\quad \hbox{for $x\in (0,\ell_\me)$ and all } \me\in \mE,
\]
and of
\begin{equation}\label{eq:pbinBv}
B_\mv \gamma_\mv( u )-P_\mv^{(d,0)} x_\mv=\mg_\mv \quad\text{for all }\mv\in \mV.
\end{equation}
Such a solution is given by
\[
u_\me(x)=K_\me+u_\me^{\rm nh}(x)\quad \hbox{for all } x\in [0,\ell_\me],
\me\in \mE,
\]
with $K_\me\in \C^{k_\me}$ and where
\[u_\me^{\rm nh}(x)=
\int_0^xM^{-1}_\me(y)\mf_\me(y)\,dy\quad
\hbox{for all } x\in [0,\ell_\me], \me\in \mE.
\]
It remains to fix the vectors $K_\me$. For that purpose, we recall (see \cite[\S3]{KraMugNic20}) that the condition $\gamma_\mv(\mmu)\in Y_\mv$ at any vertex
$\mv\in \mV$ is equivalent to
\be\label{eq:mainbcequiv2nonh}
(K_\me)_{\me\in \mE}\cdot \overline{\widetilde \mw^{(\mv, i)}}=-(u_\me^{\rm nh}(\mv))_{\me\in \mE}\cdot \overline{\widetilde \mw^{(\mv, i)}}
\quad \hbox{for all } i\in I_\mv.
\ee
On the other hand, problem \eqref{eq:pbinBv} is {by \eqref{eq:Yd-split} and the definition of basis,} equivalent to 
\begin{eqnarray*}
B_\mv \gamma_\mv( u )\cdot {\overline{ \my^{(\mv, j)}}}&=&\mg_\mv\cdot  {\overline{ \my^{(\mv, j)}}} \quad\text{for all } j\in  {J_\mv^{(R)}}, \mv\in \mV,\\
- P_\mv^{(d,0)}\gamma_\mv( u )\cdot  {\overline{\mw_{KB^\ast}^{(\mv, l)}}}&=&\mg_\mv\cdot  {\overline{ \mw_{KB^\ast}^{(\mv, l)}}} \quad\text{for all } l\in J_\mv^{(K)}, \mv\in \mV,
\end{eqnarray*}
and hence to}
\begin{eqnarray}\label{eq:pbinBvequiv}
(K_\me)_{\me\in \mE}\cdot \overline{{\widetilde \mw_{RB^\ast}^{(\mv, j)}}}&=&\mg_\mv\cdot  {\overline{ \widetilde\my^{(\mv, j)}}}
-(u_\me^{\rm nh}(\mv))_{\me\in \mE}\cdot \overline{{\widetilde \mw_{RB^\ast}^{(\mv, j)}}} \quad \text{for all } j\in  {J_\mv^{(R)}}, \mv\in \mV,
\\
 (K_\me)_{\me\in \mE}\cdot  {\overline{{\widetilde  \mw_{KB^\ast}^{(\mv, l)}}}}&=&- \mg_\mv\cdot  {\overline{{\widetilde \mw_{KB^\ast}^{(\mv, l)}}}}
 -(u_\me^{\rm nh}(\mv))_{\me\in \mE}\cdot  {\overline{ {\widetilde \mw_{KB^\ast}^{(\mv, l)}}} }\quad\text{for all } l\in J_\mv^{(K)}, \mv\in \mV. \label{eq:pbinBvequiv2}
\end{eqnarray}
{By \eqref{eq:basis} it follows} that \eqref{eq:mainbcequiv2nonh}-\eqref{eq:pbinBvequiv}-\eqref{eq:pbinBvequiv2} is a 
{$k\times k$ linear system in $(K_\me)_{\me\in \mE}$ that has a unique solution.
This shows that the operator
$\mathbb A_0$ 
is an isomorphism from $D(\mathbb A)$ into ${\bL}^2_{d}(\mathcal G)$ and, in particular, it is closed.
Hence, by dissipativity of $\mathbb A_0-\lambda {\mathbb I}$, it is also quasi-m-dissipative. We conclude that $\mathbb A_0$, and hence also $\mathbb A$, generate strongly continuous semigroup on $\bL_d^2(\mathcal G)$.}
 \end{proof}

Repeating the same argument for $-\mathbb A$ yields the following.

\begin{cor}\label{cor:group}
For all $\mv\in \mV$, let~\eqref{eq:basis} hold and let moreover $Y_\mv$ be a subspace of the null isotropic cone of the quadratic form associated with $T_\mv +Q_\mv B_\mv+B_\mv^\ast Q_\mv$. 
 Then $\mathbb A$ generates a strongly continuous group on $\bL^2_d(\mathcal G)$.
\end{cor}

\begin{rem}
Because $\dim Y^{(d)}\le \dim Y\le 2k<\infty$, the compact embedding of each $H^1(0,\ell_\me)$ in $L^2(0,\ell_\me)$, and hence of $\bigoplus_{\me\in\mE}H^1(0,\ell_\me)$ in  $\bigoplus_{\me\in\mE}L^2(0,\ell_\me)$, directly implies that $\mathbb A$ has compact resolvent, regardless of the imposed transmission conditions at the vertices.
\end{rem}

\begin{rem}\label{rem-Z and global}
(1) Formula~\eqref{eq:Adissdyn-2} shows that, in order to obtain dissipativity (rather than mere \textit{quasi}-dissipativity) of $\mathbb A$ on $\bL^2_d(\mathcal G)$, hence generation of a \textit{contractive} semigroup, the assumptions of Theorem~\ref{prop:main1dyn} shall be complemented by the following:
\begin{itemize}
\item $Q_\me(x)N_\me(x)+N_\me(x)^\ast Q_\me(x)-(Q_\me M_\me)'(x)$ is negative semi-definite, for all $\me\in\mE$ and a.e.\ $x\in (0,\ell_\me)$; and 
\item $Y^{(d)}_\mv$ is  for all $\mv\in \mV$ a subspace of the negative isotropic cone of the quadratic form associated with $Q_\mv C_\mv+C^\ast_\mv Q_\mv$.
\end{itemize}
(2) If, additionally to the assumptions of \autoref{cor:group}, 
\begin{itemize}
\item $Q_\me(x)N_\me(x)+N_\me(x)^\ast Q_\me(x)=(Q_\me M_\me)'(x)$, for all $\me\in\mE$ and a.e.\ $x\in (0,\ell_\me)$; and
\item $Y^{(d)}_\mv$ is  for all $\mv\in \mV$ a subspace of the null isotropic cone of the quadratic form associated with $Q_\mv C_\mv+C^\ast_\mv Q_\mv$,
\end{itemize}
then $\mathbb A$ generates in fact a \textit{unitary} group on $\bL^2_d(\mathcal G)$. 

In both cases, the quadratic form on $Y^{(d)}_\mv$ is considered with respect to the Euclidean inner product.
Observe, however, that both contractivity and unitarity -- hence decay or conservation of (an appropriate notion of) energy -- hold of course, under the above assumptions,  with respect to the equivalent norm of $\bL^2(\mathcal G)\oplus Y^{(d)}$ defined in~\eqref{prod-d}, which depends on the matrices $Q_\me(x)$ and $Q_\mv$, $x\in (0,\ell_\me)$, $\me\in\mE$, $\mv\in \mV$. 
\end{rem}

{
It turns out that the condition~\eqref{eq:basis} is not satisfied in some relevant applications, see e.g.\ Section~\ref{sec:dirac}. We present a different approach that requires proving the dissipativeness of both $\mathbb A$ and its adjoint $\mathbb A^\ast$. To begin with, let us elaborate on some ideas presented in~\cite[\S3]{KraMugNic20} and describe $\mathbb A^\ast$.

\begin{lemma}\label{lem:adjoint}
The adjoint of the operator $\mathbb A$ is given by
 \begin{equation*}
\begin{split}
D({\mathbb A}^\ast)&=\left\{\begin{pmatrix} v \\ \mmy \end{pmatrix}\in {\bL}^2_{d}(\mathcal G):
v\in D_{\max} \hbox{ such that} \begin{pmatrix} \gamma_\mv(v) \\ \mmy_\mv \end{pmatrix}\in \mathbb Y^\ast_\mv \hbox{ for all }\mv\in \mV\right\},\\
{\mathbb A}^\ast&=\begin{pmatrix}
{\mathcal A}^\ast & 0\\
\widetilde{\mathcal{B}} & \widetilde { \mathcal C}
\end{pmatrix},
\end{split}
\end{equation*}
where
 \begin{equation*}
\begin{split}
({\mathcal A}^\ast v)_\me& :=- M_\me  v'_\me  -Q_\me^{-1}\left( Q_\me M_\me  \right)'   v_\me
+Q_\me^{-1}N_\me^\ast Q_\me v_\me,\qquad \me\in\mE,
\end{split}
\end{equation*}
and
\[
\begin{split}
(\widetilde{\mathcal{B}}  v)_\mv&:= Q_\mv^{-1}P_{\mv}^{(d)}  
T_\mv \gamma_\mv ( v), \quad \mmv\in \mV,
\\
(\widetilde{\mathcal{C}}  v)_\mv&:=Q_\mv^{-1}P_{\mv}^{(d)}  B_\mv^\ast    Q_\mv  \mmy_\mv  + Q_\mv^{-1} C_\mv^\ast  Q_\mv  \mmy_\mv, \quad \mmv\in \mV,
\end{split}
\]
and, finally, the subspace $\mathbb Y^\ast_\mv$  of $\C^{k_\mv}\oplus Y_\mv^{(d)}$ is defined by
\begin{equation}\label{eq:zv}
\mathbb Y^\ast_\mv:=\Ker \begin{pmatrix} P^{(d),\perp}_\mv T_\mv &  P^{(d),\perp}_\mv B^\ast_\mv Q_\mv\end{pmatrix},
\end{equation}
where $P_{\mv}^{(d),\perp}$ is the orthogonal projector onto $(Y_\mv^{(d)})^{\perp_y}$ with respect to the Euclidean inner product.\end{lemma}

\begin{proof}
 First we notice that $D(\mathbb A)$ is dense. Indeed, given $\begin{pmatrix} g \\ \mmh \end{pmatrix}\in {\bL}^2_{d}(\mathcal G)$, by the surjectivity of the trace mapping, there exists
 $u\in D_{\max}$
such that
\[
\mmh=P_\mv^{(d)} \gamma_\mv(u),
\]
and $\gamma_\mv(u)\in Y_\mv$, for all $\mv\in \mV$.
This in particular means that the pair $\begin{pmatrix} u \\ \mmh \end{pmatrix}\in D(\mathbb A)$.
Now,  since $g-u\in{\bL}^2(\mathcal G)$, there exists a sequence of elements 
$\varphi^{(n)}\in \bigoplus_{\me\in \mE} {\mathcal D}(0,\ell_\me)^{k_\me}$ such that
\[
\varphi^{(n)}\to g-u \hbox{ in } {\bL}^2(\mathcal G).
\]
Since $\begin{pmatrix} \varphi^{(n)} \\ 0\end{pmatrix}$ belongs trivially to $D(\mathbb A)$, we get that $\begin{pmatrix} u+\varphi^{(n)} \\ \mmh\end{pmatrix}$ belongs  to $D(\mathbb A)$
and satisfies
\[
\begin{pmatrix} u+\varphi^{(n)} \\ \mmh\end{pmatrix} \to \begin{pmatrix} g \\ \mmh\end{pmatrix}\hbox{ in } {\bL}^2_d(\mathcal G).
\]
 
By definition,
 $\begin{pmatrix} v \\ \mmy \end{pmatrix}\in {\bL}^2_{d}(\mathcal G)$ belongs to $D(\mathbb A^\ast)$
if and only if there exists 
$\begin{pmatrix} g \\ \mmh \end{pmatrix}\in {\bL}^2_{d}(\mathcal G)$ such that
\[
\left(\mathbb A  \begin{pmatrix} u \\ \mmx \end{pmatrix} , \begin{pmatrix} v \\ \mmy \end{pmatrix}\right)_d=\left(\begin{pmatrix} u \\ \mmx \end{pmatrix}, \begin{pmatrix} g \\ \mmh \end{pmatrix}\right)_d\quad\hbox{for all } \begin{pmatrix} u \\ \mmx \end{pmatrix}\in D(\mathbb A)
\]
and in such a case
\[
{\mathbb A}^\ast v=\begin{pmatrix} g \\ \mmh \end{pmatrix}.
\]

Taking first $ \mmx=0$ and $u_\me\in \mathcal{D}(0,\ell_\me)$ (which yields a pair
$\begin{pmatrix} u \\ \mmx \end{pmatrix}\in D(\mathbb A)$) we find  that
\begin{equation}\label{DEadjoint}
-Q_\me M_\me  v'_\me  -\left( Q_\me M_\me  \right)'   v_\me
+N_\me^\ast Q_\me v_\me=Q_\me g_\me 
\end{equation}
holds  in the distributional sense, hence $v$ belongs to $D_{\max}$.
We can thus apply  the  identity
\begin{equation}\label{eq:Auv}
\begin{split}
\left(\mathcal A u, v\right)
&=
\sum_{\me\in\mE}\int_0^{\ell_\me}
  u_\me \cdot 
\overline{\left(-Q_\me M_\me  v'_\me  -\left( Q_\me M_\me  \right)'   v_\me
+N_\me^\ast Q_\me v_\me\right)}  \,dx
  \\
&\qquad +
\sum_{\mv\in\mV}  T_\mv \gamma_\mv(u)\cdot\gamma_\mv (\bar v)
 \qquad \hbox{for all } u, v \in D_{\max}
\end{split}
\end{equation}
{(see the proof of \cite[Lem.~3.10]{KraMugNic20}).
By \eqref{DEadjoint}, 
 the definition of $\mathbb A$, and inner product \eqref{prod-d}, we obtain}
\[
\sum_{\mv\in\mV}  T_\mv \gamma_\mv(u)\cdot\gamma_\mv (\bar v)
+\sum_{\mv\in\mV} \left(
Q_\mv \left( {B_\mv \gamma_\mv(u) +C_\mv \mmx_\mv} \right) \cdot \bar \mmy_\mv
\right)
=\sum_{\mv\in\mV} 
Q_\mv \mmx_\mv\cdot \bar \mmh_\mv,  \qquad \hbox{for all } \begin{pmatrix} u \\ \mmx \end{pmatrix}\in D(\mathbb A).
\]
As $\mmx_\mv=P_\mv^{(d)}  \gamma_\mv(u)$, we further have
\[
\sum_{\mv\in\mV}  T_\mv \gamma_\mv(u)\cdot\gamma_\mv (\bar v)
+\sum_{\mv\in\mV} \left(
Q_\mv \left(B_\mv +C_\mv P^{(d)}_\mv\right)\gamma_\mv(u) \cdot  \bar \mmy_\mv
\right)
=\sum_{\mv\in\mV} 
Q_\mv P_\mv^{(d)}  \gamma_\mv(u)\cdot \bar \mmh_\mv,  \qquad \hbox{for all } { u\in D(\mathcal A),}
\]
that we write equivalently as
\[
\sum_{\mv\in\mV}  
 \gamma_\mv(u)\cdot \overline{\left(
T_\mv \gamma_\mv ( v)
+  \left(B_\mv^\ast  +  C_\mv^\ast\right)  Q_\mv   \mmy_\mv
- Q_\mv \mmh_\mv\right)}=0,  \qquad \hbox{for all }{ u\in D(\mathcal A).}
\]
By the surjectivity of the trace mapping, since $\gamma_\mv(u)\in Y_\mv$, we find that
\begin{equation}\label{eq:pr0}
P_{Y_\mv} \left(
T_\mv \gamma_\mv ( v)
+  \left(B_\mv^\ast  +  C_\mv^\ast\right)  Q_\mv  \mmy_\mv
- Q_\mv \mmh_\mv\right)=0,
\end{equation}
where $P_{Y_\mv}$ is the orthogonal projector on $Y_\mv$ with respect to the Euclidean inner product.

Since $Y_\mv=Y_\mv^{(d)}\oplus (Y_\mv^{(d)})^{\perp_y}$ (orthogonal sum), and since $C_\mv^\ast  Q_\mv  \mmy_\mv
- Q_\mv \mmh_\mv$ belongs to $Y_\mv^{(d)}$, \eqref{eq:pr0} is equivalent to
\begin{equation}\label{eq:bcA*}
P_{\mv}^{(d),\perp} \left(
T_\mv \gamma_\mv ( v)
+  B_\mv^\ast    Q_\mv  \mmy_\mv\right)=0,
\end{equation}
and
\begin{equation}\label{eq:defh}
P_{\mv}^{(d)} \left(
T_\mv \gamma_\mv ( v)
+  B_\mv^\ast    Q_\mv  \mmy_\mv\right) +  C_\mv^\ast  Q_\mv  \mmy_\mv
- Q_\mv \mmh_\mv=0.
\end{equation}
Finally, we notice that \eqref{eq:bcA*} means equivalently that $\begin{pmatrix} \gamma_\mv(v) \\ \mmy_\mv \end{pmatrix}\in \mathbb Y^\ast_\mv$.
On the other hand, \eqref{eq:defh} defines
$h_\mv$, namely, it is equivalent to
\[
\mmh_\mv=
Q_\mv^{-1}P_{\mv}^{(d)} \left(
T_\mv \gamma_\mv ( v)
+  B_\mv^\ast    Q_\mv  \mmy_\mv\right) + Q_\mv^{-1} C_\mv^\ast  Q_\mv  \mmy_\mv.
\]
This concludes the proof.
 \end{proof}

\begin{rem}
Observe that~\eqref{eq:bcA*} is a property similar to $\mmx=P_{\mv}^{(d)}  \gamma_\mv ( u)$ and 
to the boundary condition $\gamma_\mv ( u)\in Y_\mv$, since these two conditions can be compactly written
\begin{equation}\label{eq:comp-serge}
P_{\mv}^{(d),\perp\perp}(\gamma_\mv(u)-\mmx)=0,
\end{equation}
where $P_{\mv}^{(d),\perp\perp}$ means  the orthogonal projector on the orthogonal of $(Y_\mv^{(d)})^\perp$ in $\C^{k_\mv}$ (equal to $Y_\mv^{(d)}\oplus Y_\mv^\perp$)
with respect to the Euclidean inner product.
Indeed, \eqref{eq:comp-serge} means that
\[
\gamma_\mv(u)-\mmx\in (Y_\mv^{(d)})^\perp,
\]
or, equivalently,
\[
\gamma_\mv(u)=\mmx+\mmy
\]
with $\mmy\in
(Y_\mv^{(d)})^\perp$. This gives $\gamma_\mv(u)\in Y_\mv$
and taking the projection on $Y_\mv^{(d)}$ that $\mmx=P_{\mv}^{(d)}  \gamma_\mv ( u)$. 

If in particular $Y^{(d)}_\mv=\{0\}$ and hence $B^\ast_\mv=0$ and the range of $P^{(d),\perp}_\mv$ is $Y^\perp_\mv$, the assertion in \autoref{lem:adjoint} thus agrees with~\cite[Lemma~3.10]{KraMugNic20}.
\end{rem} 
We are finally in the position to propose a set of sufficient conditions for well-posedness different from those in \autoref{prop:main1dyn} and \autoref{cor:group}.

\begin{theo}\label{thm:new-for-dirac}
For all $\mv\in \mV$, let 
\begin{itemize}
\item $Y^{(d)}_\mv$ be a subspace of the nonpositive isotropic cone of the quadratic form on $Y^{(d)}_\mv$ associated with 
\[T_\mv +Q_\mv B_\mv+B_\mv^\ast Q_\mv -\lambda P_\mv^{(d)} Q_\mv  P_\mv^{(d)}\]
 for some $\lambda\geq 0$, and
\item $\mathbb Y^\ast_\mv$ as in~\eqref{eq:zv} be a subspace of the nonpositive isotropic cone (with respect to the Euclidean inner product in $\C^{k_\mv}\oplus Y^{(d)}_\mv$) of the quadratic form associated with
\begin{equation*}
\begin{pmatrix}
-T_\mv-2\mu & T_\mv \\
P_{\mv}^{(d)}T_\mv & (P_{\mv}^{(d)}B_\mv^\ast  -\mu \Id )  Q_\mv+Q_\mv( B_\mv -\mu \Id)
\end{pmatrix}
\end{equation*}
for some $\mu\ge 0$.
\end{itemize} 
Then $\mathbb A$ is a quasi-$m$-dissipative operator.
In particular, $\mathbb A$ generates a strongly continuous semigroup on $\bL^2_d(\mathcal G)$.
\end{theo}

\begin{proof}
We already know that $\mathbb A$ is densely defined. Also, it is not difficult to prove that $\mathbb A$ is closed: this can be seen invoking~\cite[Lemma~2.3]{KraMugNag03}, since closedness of $\mathcal A$ has been already observed in~\cite{KraMugNic20}, based on computations in~\cite{BasCor16}.

By~\cite[Cor.~II.3.17]{EngNag00}, $m$-dissipativity of $\mathbb A$ will follow if we can check that both $\mathbb A$ and its adjoint $\mathbb A^\ast$ are dissipative. Similarly to what we have already done in \autoref{prop:main1dyn}, for the sake of simplicity and  without loss of generality we assume in the following that $N_\me=C_\mv=0$.

The proof of \autoref{prop:main1dyn} shows that $\mathbb A$ is dissipative under our assumptions. In order to check dissipativity of $\mathbb A^\ast$, we start from the identity \
\begin{equation}\label{eq:A*diss}
\begin{split}
\Re\left({\mathcal A}^\ast u, u\right)
&=
\Re\sum_{\me\in\mE}\int_0^{\ell_\me}
\left(-\left( Q_\me M_\me  \right)'   u_\me
+N_\me^\ast Q_\me u_\me\right)\cdot  \bar u_\me \,dx\\
&\quad +\frac{1}{2}  \sum_{\me\in\mE}\int_0^{\ell_\me}
  \left( Q_\me M_\me\right)' u_\me \cdot      \bar u_\me  \,dx
-\frac{1}{2}\sum_{\mv\in\mV}  T_\mv \gamma_\mv(u)\cdot\gamma_\mv(\bar u)
\end{split}
\end{equation}
which was derived in the proof of~\cite[Thm.~3.11]{KraMugNic20} for all $u\in D_{\max}$. We then find that  for all $\mathfrak u=(u,\mmx)^\top \in D({\mathbb A}^\ast)$,
\begin{equation}\label{eq:A*dyndiss}
\begin{split}
\Re\left({\mathbb A}^\ast \mathfrak{u}, \mathfrak{u}\right)_d
&=
\Re\sum_{\me\in\mE}\int_0^{\ell_\me}
\left(-\left( Q_\me M_\me  \right)'   u_\me
+N_\me^\ast Q_\me u_\me\right)\cdot  \bar u_\me \,dx\\
&\quad +\frac{1}{2}  \sum_{\me\in\mE}\int_0^{\ell_\me}
  \left( Q_\me M_\me\right)' u_\me \cdot      \bar u_\me  \,dx
-\frac{1}{2}\sum_{\mv\in\mV}  T_\mv \gamma_\mv(u)\cdot\gamma_\mv(\bar u)
\\
 &+\Re\sum_{\mv\in\mV}\left(P_{\mv}^{(d)} \left(
T_\mv \gamma_\mv ( u)
+  B_\mv^\ast    Q_\mv  \mmx_\mv\right) \right)\cdot \bar\mmx_\mv. 
\end{split}
\end{equation} 
Hence, ${\mathbb A}^\ast$ is quasi-dissipative if for some $\mu\ge 0$ it holds
\[
-\frac{1}{2}T_\mv \xi \cdot \bar \xi
+\Re\left(
\left(P_{\mv}^{(d)} \left(
T_\mv \xi
+  B_\mv^\ast    Q_\mv  \mmx\right)\right)\cdot \bar\mmx\right)\leq {\mu} \left\|\begin{pmatrix} \xi\\ Q_\mv^\frac12 \mmx \end{pmatrix} \right\|^2_{\C^{k_\mv}\oplus Y_\mv^{(d)}}
\quad\hbox{for all } \begin{pmatrix} \xi\\ \mmx \end{pmatrix}\in \mathbb Y^\ast_\mv,\]
where the inner product and the norm are the Euclidean ones. This is
equivalent to
\[
\left(
\begin{pmatrix}
-T_\mv & T_\mv \\
P_{\mv}^{(d)}T_\mv & P_{\mv}^{(d)}B_\mv^\ast    Q_\mv+Q_\mv  B_\mv 
\end{pmatrix}
\begin{pmatrix}
\xi \\ \mmx
\end{pmatrix},
\begin{pmatrix}
\xi \\ \mmx
\end{pmatrix}\right)_{\C^{k_\mv}\oplus Y_\mv^{(d)}}
\leq 2\mu \left\|\begin{pmatrix} \xi\\ Q_\mv^\frac12\mmx \end{pmatrix} \right\|^2_{\C^{k_\mv}\oplus Y_\mv^{(d)}}
\quad\hbox{for all } \begin{pmatrix} \xi\\ \mmx \end{pmatrix}\in \mathbb Y^\ast_\mv,
\]
and the claim follows.
\end{proof} 

Again, repeating the same argument for $-\mathbb A$ yields the following.

\begin{cor}\label{cor:group-2}
For all $\mv\in \mV$, let 
\begin{itemize}
\item $Y^{(d)}_\mv$ be a subspace of the null isotropic cone of the quadratic form on $Y^{(d)}_\mv$ associated with 
\[T_\mv +Q_\mv B_\mv+B_\mv^\ast Q_\mv -\lambda P_\mv^{(d)} Q_\mv  P_\mv^{(d)}\]
 for some $\lambda\geq 0$, and
\item $\mathbb Y^\ast_\mv$ as in~\eqref{eq:zv} be a subspace of the null isotropic cone of the quadratic form  on $\C^{k_\mv}\oplus Y^{(d)}_\mv$ associated with
\begin{equation*}
\begin{pmatrix}
-T_\mv-2\mu & T_\mv \\
P_{\mv}^{(d)}T_\mv & (P_{\mv}^{(d)}B_\mv^\ast  -\mu \Id )  Q_\mv+Q_\mv( B_\mv -\mu \Id)
\end{pmatrix}
\end{equation*}
for some $\mu\ge 0$.
\end{itemize} 
Then both $\pm\mathbb A$ are quasi-$m$-dissipative operators, and accordingly $\mathbb A$ generates a strongly continuous group on $\bL^2_d(\mathcal G)$.
\end{cor}

\begin{rem}\label{rem-Z and global-2}
We can formulate conditions for dissipativity (rather than mere \textit{quasi}-dissipativity) and unitarity of the (semi)group generated by $\mathbb A$ along the lines of \autoref{rem-Z and global}.

(1) $\mathbb A$ generates a \textit{contractive} semigroup on $\bL^2_d(\mathcal G)$ if the assumptions of \autoref{thm:new-for-dirac} are complemented by the following:
\begin{itemize}
\item $Q_\me(x)N_\me(x)+N_\me(x)^\ast Q_\me(x)-(Q_\me M_\me)'(x)$ is negative semi-definite, for all $\me\in\mE$ and a.e.\ $x\in (0,\ell_\me)$; and 
\item $Y^{(d)}_\mv$ is for all $\mv\in \mV$ a subspace of the negative isotropic cone of the quadratic form associated with $Q_\mv C_\mv+C^\ast_\mv Q_\mv$.
\end{itemize}
(2) If, additionally to the assumptions of \autoref{cor:group-2}, 
\begin{itemize}
\item $Q_\me(x)N_\me(x)+N_\me(x)^\ast Q_\me(x)=(Q_\me M_\me)'(x)$, for all $\me\in\mE$ and a.e.\ $x\in (0,\ell_\me)$;
\item $Y^{(d)}_\mv$ is for all $\mv\in \mV$ a subspace of the null isotropic cone of the quadratic form associated with $Q_\mv C_\mv+C^\ast_\mv Q_\mv$; and
\item $\mathbb Y^\ast_\mv$ is for all $\mv\in \mV$ a subspace of the null isotropic cone of the quadratic form associated with
\begin{equation*}
\begin{pmatrix}
0 & 0\\
0 & Q_\mv C_\mv+C^\ast_\mv Q_\mv
\end{pmatrix},
\end{equation*}
then $\mathbb A$ generates a \textit{unitary} group on $\bL^2_d(\mathcal G)$.
\end{itemize}
\end{rem}
}

\begin{rem}\label{rem:bcglob}
 We can further easily replace local boundary conditions by global ones: to this purpose, we take the $2k\times 2k$ matrix $T$ given by
\begin{equation}\label{eq:Tdef}
T:= \begin{pmatrix}
-\diag \left(
Q_\me(0) M_\me(0)\right)_{\me\in\mE}& 0
\\
0& \diag \left(Q_\me(\ell_\me) M_\me(\ell_\me)\right)_{\me\in\mE}
\end{pmatrix}
\end{equation}
and replace $B_\mv, C_\mv, Q_\mv$ by globaly defined operators $B\colon Y\to Y^{(d)}$, $C^{(d)}, Q^{(d)} \colon Y^{(d)}\to Y^{(d)}$ for 
some subspaces $Y^{(d)}\subset Y\subset \C^{2k}$. {With the notation
\[
\gamma (u):= 
\left(\left(u_\me(0)\right)_{\me\in\mE}, \left(u_\me (\ell_\me)\right)_{\me\in\mE}\right)^\top,
\]
we} thus consider operator $\mathbb A$ defined as in \eqref{eq:opA} with domain
\begin{eqnarray}\label{eq:domAglobal}
 D(\mathbb A):=\left\{\begin{pmatrix} u \\ \mmx \end{pmatrix}\in D_{\max}\oplus Y^{(d)}: 
\gamma( u )\in Y\hbox{ and }\mmx={P}^{(d)}\gamma( u )  \right\}
\end{eqnarray}
and assume $Y$ to be the appropriate isotropic cone of the quadratic form associated with $T+Q^{(d)} B + B^\ast Q^{(d)}$.
{
In this case $Z = (Y^\perp \oplus \Ran B^\ast) + \Ker B^\ast \subset \C^{2k}$ and the well-posedness condition \eqref{eq:basis} becomes 
\begin{equation}\label{eq:basis-global}
 \dim Z=\dim P_K Z = k,
\end{equation} 
where $P_K$ is the orthogonal projector onto
\[
K=\left\{\left(\left(K_\me\right)_{\me\in\mE}, \left(K_\me\right)_{\me\in\mE}\right)^\top:
K_\me\in \mathbb{C}^{k_\me} \text{ for all }\me\in\mE\right\}
\]
with respect to the euclidean inner product of $\C^{2k}$,} see  \cite[Rem.~3.13]{KraMugNic20}) for details.
In Section~\ref{sec:dirac} we are going to see that \eqref{eq:basis} and, equivalently, \eqref{eq:basis-global} may fail to hold even when the equation can be -- by other means -- proved to be well-posed.
 \end{rem}

\section{Qualitative properties}\label{sec:qualitative}

{
We now study when the (semi)group generated by $\mathbb{A}$ is, real, positive, {or $\infty$-contractive}. 
Let $C\subset \C$ be a closed and convex set; we will denote by $P_C:\C\to \C$ the projector onto $C$. 
As in \cite[\S4]{KraMugNic20}, we shall apply to the Hilbert space of $C$-valued vectors in $\bL^2_d(\mathcal G)$, i.e., to
\[
K :=\bL^2_d(\mathcal G;C):= \bL^2(\mathcal G;C) \oplus Y^{(d)}_C,
\]
a generalization (cf.~\cite[Lemma~4.3]{KraMugNic20}) of a classical result by Brezis for the invariance of the convex subsets of Hilbert spaces; here
\[
\bL^2(\mathcal G;C):=\{u\in L^2(\mathcal G): u_\me(x)\in C^{k_\me}\ \hbox{ for a.e. }x\in (0,\ell_\me)\hbox{ and all }\me \in\mE\}
\]
and
\[
Y^{(d)}_C := \{ \mmx\in Y^{(d)}:\mmx_\mv \in C^{k_\mv} \hbox{ for all }\mv\in \mV \}.
\] 
(Observe that the latter might well be trivial, like in the case of $Y^{(d)}$ spanned by the vector $(1,-1)^\top$ and $C=\R_+$.)

To this end we first need to relate the minimizing projector $\mathbb{P}_K^Q$ with respect to the inner product $(\cdot,\cdot)_d$ in the Hilbert space ${\bL}^2_{d}(\mathcal G)$ defined in \eqref{prod-d} to the minimizing projectors $P_K$ and $P^{(d)}_K$ with respect to the standard inner products in the Hilbert spaces ${\bf L}^2(\mathcal G)$ and $Y^{(d)}$, respectively: i.e., the products 
\begin{eqnarray}
 \langle u,v \rangle &:=&\sum_{\me\in\mE}\int_0^{\ell_\me} u_\me(x)\cdot\overline{v}_\me(x)\ dx, \quad u,v\in {\bf L}^2(\mathcal G), \label{st-L2}\\
\mmx \cdot \bar\mmy &:=&\sum_{\mv\in \mV} \mmx_\mv\cdot \bar \mmy_\mv , \quad \mmx,\mmy \in Y^{(d)}. \label{st-Z}
\end{eqnarray}
By following the steps in the proof of \cite[Lemma~4.4]{KraMugNic20} and performing the calculations for each component of $K$ separately, we obtain the following characterization.

\begin{lemma}\label{lem:project}
Assume $Q_\me^{\frac{1}{2}}(x)$ and $Q^{\frac{1}{2}}_\mv$ to be bijective maps on $C^{k_\me}$ and $C^{k_\mv}$ for all $\me\in\mE$ and all $x\in [0,\ell_\me]$ and for all $\mv\in\mV$, respectively.
Then the minimizing projector $\mathbb{P}_K^Q$ with respect to the inner product \eqref{prod-d} onto $K=\bL^2_d(\mathcal G;C)$
is given by
\begin{equation}\label{eq:PQKdyn}
\mathbb{P}_{K}^Q=\begin{pmatrix}Q^{-\frac{1}{2}} P_K Q^{\frac{1}{2}} & 0\\ 0 &(Q^{(d)})^{-\frac{1}{2}} P^{(d)}_K (Q^{(d)})^{\frac{1}{2}} \end{pmatrix}
\end{equation}
where $Q:=\diag(Q_\me)_{\me\in \mE}$ and $Q^{(d)}:=\diag(Q_\mv)_{\mv\in \mV}$ are  block diagonal matrices, while $P_K$ and $P^{(d)}_K$ are the minimizing projectors with respect to the standard inner products \eqref{st-L2} and \eqref{st-Z}, respectively.
\end{lemma}
In the following, we are going to focus on the cases of 
\begin{itemize}
\item $C=\R$, 
\item $C=\R_+$, 
\item  $C=\{z\in \C:|z|\le 1\}$.
\end{itemize}
Our arguments in the following rely upon~\cite[Lemma~4.3]{KraMugNic20}, which holds for quasi-m-dissipative operators; but in the first two cases ($C=\R$, $C=\R_+$) the relevant conditions for invariance are equivalent in the quasi-dissipative and dissipative case, since reality and positivity of a semigroup are not affected by a scalar additive perturbation of its generator.

To begin with, let us consider $C=\R$: then \autoref{lem:project} states that if $Q_\mv$ and $Q_\me$ are real-valued, then the minimizing projector onto $K=\bL_d^2(\mathcal G;\R)$ is given by
\begin{equation*}\label{eq:qrr}
\mathbb{P}_{K}^Q {u\choose\mmx} =
\begin{pmatrix}Q^{-\frac{1}{2}} \Re\left( Q^{\frac{1}{2}} u\right)\\ (Q^{(d)})^{-\frac{1}{2}} \Re\left( (Q^{(d)})^{\frac{1}{2}} \mmx \right)\end{pmatrix}
= {{\Re u}\choose {\Re \mmx}},\qquad  {u\choose\mmx}\in \bL^2_d(\mathcal G) .
\end{equation*}
This allows for an extension of~\cite[Prop.~4.5]{KraMugNic20}.

\begin{prop}\label{prop:real-dyn}
Under the assumptions of \autoref{prop:main1dyn} or \autoref{thm:new-for-dirac}, let
\begin{equation}\label{eq:realcond-0}
\Re \xi\in \bigoplus_{\mv\in\mV} Y_\mv \hbox{ for all } \xi\in \bigoplus_{\mv\in\mV
} Y_\mv
\quad \text{and} \quad \Re \mmx\in \bigoplus_{\mv\in\mV} Y_\mv^{(d)} \hbox{ for all } \mmx\in \bigoplus_{\mv\in\mV
} Y_\mv^{(d)},
\end{equation}
let the matrix-valued mapping $Q_\me $ be real-valued for all $\me\in \mE$, and let the matrices $Q_\mv,B_\mv,C_\mv$ be real for all $\mv\in\mV$. Then the semigroup generated by $\mathbb A$ is real if 
the matrix-valued mappings $M_\me,N_\me$ are real-valued for all $\me\in\mE$.
\end{prop}

\begin{proof}
First observe that by \cite[Lemma~4.7]{KraMugNic20}, \eqref{eq:realcond-0} holds if and only if $Y_\mv, Y_\mv^{(d)}$, for each $\mv\in\mV$,  are spanned by entry-wise real vectors only. Thus, the orthogonal projectors $P^{(d)}_\mv$ are real matrices for all $\mv$ (see, e.g., \cite[(5.13.3)]{meyer04}). 
By the assumptions we then obtain, 
\[\mathbb{P}_K^Q {u\choose\mmx}\in D(\mathbb{A})\text{ whenever }{u\choose\mmx} \in D(\mathbb{A}).\]
As in the proof of \cite[Prop.~4.5]{KraMugNic20} we deduce that the reality of the semigroup is equivalent to 
\begin{equation}\label{eq:condreal}
\left(\mathbb{A} {{\Re u}\choose {\Re\mmx}},{{\Im u}\choose {\Im \mmx}}\right)_d =\left( {\mathcal A{\Re u}\choose \mathcal B{\Re u}+\mathcal C{\Re\mmx}},{{\Im u}\choose {\Im \mmx}}\right)_d\in \mathbb{R} \quad \text{for all } {u\choose\mmx} \in D(\mathbb{A}),
\end{equation}
using the notation from~\eqref{eq:opA}.

Now,  the first term reads $\sum_{\me\in\mE}\int_0^{\ell_\me}Q_\me (M_\me \frac{d}{dx}+N_\me) \Re u\cdot \Im \bar u\ dx\in\R$ for all $u\in D(\mathcal A)$, which by \cite[Lemma~4.6]{KraMugNic20} is the case
if and only if $M_\me,N_\me$ are real-valued for all $\me\in\mE$. The boundary term  
$$ \sum_{\mv\in \mV} Q_\mv \left(B_\mv+C_\mv P^{(d)}_\mv\right) \gamma_\mv(\Re u) \cdot { P^{(d)}_\mv \gamma_\mv(\Im \bar u )} \in\R$$ if and only if $B_\mv+C_\mv P^{(d)}_\mv $ is real for all $\mv\in\mV$, since all entries of $Q_\mv,P^{(d)}_\mv$ are real. Finally, the reality of $B_\mv,C_\mv$ is  sufficient to ensure the reality of $B_\mv+C_\mv P^{(d)}_\mv$.
\end{proof}

We continue with the study of positivity.
Without loss of generality we restrict ourselves to the real Hilbert space $\bL^2_d(\mathcal G;\R)$ and consider the convex subsets $C=\R_+$. First, let us recall that by \cite[Lemma~4.8]{KraMugNic20}, a real symmetric and positive definite matrix is a lattice isomorphism if and only if it is diagonal. Therefore we shall assume that the matrices $Q_\mv$ and $Q_\me(x)$ are real and diagonal for all $\mv\in\mV$, $\me\in\mE$, and $x\in [0,\ell_\me]$. Therefore the minimizing projector $\mathbb{P}_K^Q$ onto $K=\bL^2_d(\mathcal G;\R_+)$ given in \autoref{eq:PQKdyn} again takes a simpler form,
\begin{equation*}
\mathbb{P}_K^Q {u\choose\mmx} =
\begin{pmatrix}Q^{-\frac{1}{2}} \left( Q^{\frac{1}{2}} u\right)^+\\ (Q^{(d)})^{-\frac{1}{2}} \left( (Q^{(d)})^{\frac{1}{2}} \mmx \right)^+\end{pmatrix}
= {{u^+}\choose{\mmx^+}} .
\end{equation*}

\begin{prop}\label{prop:posit-dyn}
Under the assumptions of \autoref{prop:main1dyn} or \autoref{thm:new-for-dirac}, 
let the matrices
\begin{itemize}
\item $N_\me(x), Q_\me(x)$, $Q_\mv, B_\mv,C_\mv$ be real-valued,
\item $M_\me(x), Q_\me(x), Q_\mv$ be diagonal, and
\item the projector $P^+_\mv$ onto the positive cone of $\R^{k_\mv}$ commutes with $P_\mv^{(d)}$,
\end{itemize}
for all $\me\in\mE$, a.e.\ $x\in [0,\ell_\me]$, and all $\mv\in\mV$. Furthermore, let
\begin{equation}\label{eq:cond-positive-0}
\xi^+ \in {\bigoplus_{\mv\in\mV} Y_\mv } \hbox{ for all } \xi\in {\bigoplus_{\mv\in\mV} Y_\mv }
{\quad \text{and} \quad  \mmx^+\in \bigoplus_{\mv\in\mV} Y_\mv^{(d)} \hbox{ for all } \mmx\in \bigoplus_{\mv\in\mV
} Y_\mv^{(d)}.}
\end{equation}
If, additionally, all matrices $B_\mv+C_\mv P^{(d)}_\mv$ are positive, then the semigroup generated by $\mathbb A$ on $\bL^2_d(\mathcal G, \mathbb R)$ is positive if for all $\me\in\mE$ and a.e.\ $x\in [0,\ell_\me]$  all off-diagonal entries of the matrices
$N_\me (x)$   are nonnegative.
In the special case of $B_\mv=0$ for all $\mv\in \mV$, 
the semigroup generated by $\mathbb A$ is positive if 
all off-diagonal entries of the matrices $N_\me (x)$ and {$C_\mv$} are  nonnegative, for all $\me\in\mE$, a.e.\ $x\in [0,\ell_\me]$, and all $\mv\in \mV$.
\end{prop}

We stress that nonnegativity of the off-diagonal entries of $N_\me$ and {$C_\mv$} amounts to asking that the semigroups generated by $N_\me$ and $C_\mv$ are both positive.

\begin{proof}
Also in this case, it follows from the assumptions that 
\[\mathbb{P}_K^Q {u\choose\mmx}\in D(\mathbb{A})\text{ whenever }{u\choose\mmx} \in D(\mathbb{A}).
\]
By repeating the arguments in the proof of \cite[Prop.~4.9]{KraMugNic20} we obtain that the semigroup is positive if and only if
\begin{equation}\label{eq:condpos}
\left( {\mathbb A} {{u^+}\choose{\mmx^+}}, {{u^-}\choose{\mmx^-}}\right)_d \ge 0\quad \text{for all } {u\choose\mmx} \in D(\mathbb{A}).
\end{equation}
We are going to consider the two components separately. For the first one we have that, by \cite[Lemma~4.11]{KraMugNic20}, $( \mathcal A u^+, u^-) \ge 0$ 
if and only the matrices $M_\me (x)$ are diagonal and all off-diagonal entries of the matrices $N_\me (x)$ are {nonnegative}.
Let us turn to the second component: by surjectivity of $\gamma_\mv:D(\mathcal A)\to Y_\mv$ {and \eqref{eq:QvYd}}, nonnegativity of the boundary term
$ \sum_{\mv\in \mV} Q_\mv \left(B_\mv+C_\mv P^{(d)}_\mv\right) \gamma_\mv(u^+) \cdot { P^{(d)}_\mv \gamma_\mv(u^- )}$ for all $u\in D(\mathcal A)$ is equivalent to
\begin{equation}\label{eq:positC1}
 \sum_{\mv\in \mV} Q_\mv \left(B_\mv+C_\mv P^{(d)}_\mv\right) {\mmy^+_\mv \cdot { \mmy^-_\mv }}\ge 0\quad\hbox{for all }{\mmy} \in Y_\mv;
\end{equation}
  or,  in the special case $B=0$, to
\begin{equation}\label{eq:positC2} 
\sum_{\mv\in \mV} Q_\mv C_\mv \mmx^+_\mv \cdot \mmx^-_\mv \ge 0\quad\hbox{for all }\mmx\in Y^{(d)}_\mv.
\end{equation}
Now, \eqref{eq:positC1} certainly holds whenever $B_\mv+C_\mv P^{(d)}_\mv$ is a positive matrix. On the other hand, by~\cite[Thm.~2.6]{Ouh05} \eqref{eq:positC2}  is equivalent to positivity of the semigroup generated by {$C_\mv$}, i.e., to the condition that the real matrix $C_\mv$ has nonnegative off-diagonal entries.
\end{proof}

Let us finally address the question whether our semigroup is $\infty$-contractive: this is a natural issue, since the prototypical example of a hyperbolic equation -- the transport equation on $\R$ -- is governed by a semigroup of isometries on $L^p(\R)$ for all $p\in [1,\infty]$.
To this aim, let us introduce the Lebesgue-type spaces
\[
{\bL}^p_{d}(\mathcal G):=\bL^p(\mathcal G)\oplus Y^{(d)},\qquad p\in [1,\infty],
\]
equipped with the canonical $p$-norm.  

\begin{prop}\label{prop:infty-contr-dyn}
Assume our standing \autoref{assum-MQN}-\ref{assum-spaces} hold with $Q_\me,Q_\mv$ identity matrices. 
Under the assumptions of \autoref{prop:main1dyn} and~\autoref{rem-Z and global}.(1) {or else of \autoref{thm:new-for-dirac} and \autoref{rem-Z and global-2}.(1)}, let for all $\me\in\mE$, all $x\in [0,\ell_\me]$, and all $\mv\in\mV$ the matrices 
\begin{itemize}
\item $M_\me(x)$ be diagonal and
\item $N_\me (x)$ generate semigroups on $L^2(\mathcal G)$ that are contractive with respect to the $\infty$-norm.
\end{itemize}
Furthermore, let
\begin{equation}\label{eq:cond-infty-contr}
(1\wedge |\xi|)\sign \xi \in {\bigoplus_{\mv\in\mV} Y_\mv } \hbox{ for all } \xi\in {\bigoplus_{\mv\in\mV} Y_\mv }\quad \hbox{and}\quad (1\wedge |\mmx|)\sign \mmx \in {\bigoplus_{\mv\in\mV} 
Y_\mv^{ (d)} } \hbox{ for all } \mmx\in {\bigoplus_{\mv\in\mV} Y_\mv^{(d)}  }.
\end{equation}
If additionally the matrix $B_\mv+C_\mv P^{(d)}_\mv$ is $\infty$-contractive for all $\mv\in \mV$, then the semigroup generated by $\mathbb A$ on $\bL^2_d(\mathcal G)$ is $\infty$-contractive.

In the special case of $B_\mv=0$ for all $\mv\in \mV$, 
the semigroup generated by $\mathbb A$ 
is $\infty$-contractive if the semigroup generated by $C_\mv$ on $\C^{k_\mv}$ is $\infty$-contractive.
\end{prop}

Let us remind that
\[
\sign z:=\frac{z}{|z|}, \quad z\in {\C\setminus\{0\}\quad\text{and}\quad \sign 0:= 0}.  
\]
The sign of vectors with complex entries are defined accordingly.

We observe that the proof of~\cite[Lemma~6.1]{Mug07} can be easily seen to extend to our setting, where the weight matrices $Q_\me,Q_\mv$ are {identity matrices}; accordingly, the semigroups generated by $-N_\me (x)=(n_\me^{i,j}(x))_{1\le i,j\le k_\me}$ and $-C_\mv=(c_\mv^{h\ell})_{1\le h,\ell\le k_\mv}$ are $\infty$-contractive if and only if
\[
\Re n_\me^{i,i}(x)\ge \sum_{j\ne i}|n_\me^{i,j}(x)|,\qquad \Re c_\mv^{h,h}\ge \sum_{\ell \ne h}|c_\mv^{h,\ell}|\qquad \hbox{for all $\me\in\mE$, a.e.\ $x\in (0,\ell_\me)$, and all $\mv\in\mV$}.
\]
\begin{proof}
First of all, observe that $\mathbb A$ is by assumption m-dissipative in $\bL^2_d(\mathcal G)$, hence we can apply~\cite[Lemma~4.3]{KraMugNic20} in order to study invariance of the unit ball $K$ of $\bL^\infty_d(\mathcal G)$ under the semigroup generated by $\mathbb A$. Furthermore,  we can apply \autoref{lem:project}: \eqref{eq:cond-infty-contr} now guarantees that $D(\mathbb A)$ is left invariant under $\mathbb P^Q_K$
and, in view of the known formula for the minimizing projector onto the unit ball with respect to the $\infty$-norm and of~\cite[Thm.~2.13]{Ouh05}, we deduce that the relevant condition for invariance of the unit ball of $\bL^\infty_d(\mathcal G)$ under the semigroup generated by $\mathbb A$ is
\begin{equation}\label{eq:reainftycontr}
\Re\left( {\mathbb A} {{(1\wedge |u|)\sign u}\choose{(1\wedge |\mmx|)\sign \mmx}}, {{(|u|-1)^+\sign u}\choose{(|\mmx|-1)^+\sign \mmx}}\right)_d \le 0\quad \text{for all } {u\choose\mmx} \in D(\mathbb{A}).
\end{equation}
Because $\frac{d}{dx}(1\wedge |u(x)|)\sign u(x), (|\bar u(x)|-1)^+\sign \bar u(x)$ have disjoint support, by diagonality of the matrices $M_\me(x)$ one sees that
\begin{equation}\label{eq:reainftycontr-2}
\begin{split}
\Re\int_0^{\ell_\me} & \left(M_\me(x)\frac{d}{dx}+N_\me(x)\right) (1\wedge |u(x)|)\sign u(x)\cdot (|\bar u(x)|-1)^+\sign \bar u(x) dx\\
&\quad = \Re\int_0^{\ell_\me}  N_\me(x) (1\wedge |u(x)|)\sign u(x)\cdot (|\bar u(x)|-1)^+\sign \bar u(x) dx
\end{split}
\end{equation}
hence the first term in~\eqref{eq:reainftycontr} is nonpositive if the semigroup generated by the matrix $Q_\me(x)N_\me(x)$ on the unweighted space $\C^{k_\me}$ is for a.e.\ $x\in (0,\ell_\me)$ $\infty$-contractive, since in this case the integrand in the second line of~\eqref{eq:reainftycontr-2} is a negative function.

Again by~\cite[Thm.~2.13]{Ouh05}, the boundary term  in~\eqref{eq:reainftycontr} is nonpositive if in particular $B_\mv+C_\mv P^{(d)}_\mv$ is $\infty$-contractive; or more generally, cf.\ the proof of \autoref{prop:real-dyn}, if -- provided $B_\mv=0$ -- merely the semigroup generated by $C_\mv$ is $\infty$-contractive.
\end{proof}

\begin{rem}\label{rem:noinfty}
The assumption that the matrices $M_\me,Q_\me,Q_\mv$ are diagonal is very restrictive and hints at the fact that very few linear hyperbolic systems are governed by an $\infty$-contractive semigroup. This is not overly surprising: contractive semigroups on $\bL^2_d(\mathcal G)$, which are furthermore $\infty$-contractive, too, extrapolate by the Riesz--Thorin Theorem to all $\bL^p_d(\mathcal G)$-spaces, $p\ge 2$. However, Brenner's Theorem (see~\cite[Thm.~8.4.3]{AreBatHie10}) poses a serious limit to $L^p$-well-posedness of even less general systems than ours.
\end{rem}

\section{Examples}\label{sec:examples}

\subsection{Transport equation}\label{sec:transport-sikolya}
Arguably, transport equations represent the easiest setting where our Assumptions~\ref{assum-MQN} are satisfied. Transport equations
\[
{\dot{u}_\me =c_\me u'_\me}
\]
 on a network consisting of $|\mE|$ edges of unit length with transmission conditions {in $|\mV|$ vertices} given as}
 \[
 u(t,1)\in \Ran (\mathcal I^-_\omega)^\top {\quad\text{and}\quad \mathcal I^- u(t,1) = \mathcal I^+_\omega u(t,0)}
 \]
 have been introduced in~\cite{KraSik05}, where their well-posedness in an $L^1$-setting was proved. Here,  $ c_\me>0$ are constant velocity coefficients, $\mathcal I^+_\omega$ is the Kronecker product of $\mathcal I^+$ with a column stochastic $|\mV|\times |\mE|$ matrix $\mathcal W=(\omega_{\mv\me})$, and $\mathcal I^\pm$ are the signed incidence matrices introduced in~\eqref{eq:defIv}; see~\cite[\S2]{Sik05} for details. 
 It is assumed that both { signed} incidence matrices are surjective, that is of rank $|\mV|$: { by~\cite[Thm.~2.1]{BanNam14} this is the case if and only if the graph contains neither sinks nor sources.}
 As shown  in~\cite[\S3]{Sik05}, it is possible to consider dynamic conditions as well, by replacing the second (stationary) condition above by a dynamic condition of the form
\[
\frac{\partial }{\partial t}\mathcal I^- u(t,1)=\mathcal I^+_\omega u(t,0)+C\mathcal I^- u(t,1).
\]
Well-posedness of the corresponding abstract Cauchy problem was proved in~\cite[Thm.~4.5]{Sik05}. Here, we are adopting a global formalism,   assuming
\[\gamma(u) := 
\begin{pmatrix} u(1) \\ u(0)\end{pmatrix} \quad
\text{ and } \quad T: = \diag\begin{pmatrix} -\diag(c_\me)_{\me\in\mE} & 0\\ 0 & \diag(c_\me)_{\me\in\mE} \end{pmatrix}, 
\]
see~\autoref{rem:bcglob}. Note that, contrary to our notation,  in \cite{Sik05} the initial endpoint of  an edge is assumed to be in 1 and the terminal endpoint is 0. In order to be able to compare the results we stick to this terminology in the context of the present Example.

Let us show that the setting in~\cite{Sik05}  is a special case of ours: we recover the above boundary conditions letting 
\[
Y:=\Ran (\mathcal I^-_\omega)^\top \oplus\C^{|\mE|}{\simeq\C^{|\mV|}\oplus\C^{|\mE|} } \quad\hbox{and}\quad Y^{(d)} := \C^{|\mV|}{\oplus \{{\mathbf{0}}\} }
\]
as well as 
\[
B:=\begin{pmatrix} 0 &\mathcal I^+_\omega \end{pmatrix} \quad\text{and} \quad P^{(d)}:=\begin{pmatrix} \mathcal I^- & 0\end{pmatrix}. 
\]
We simply take identity matrices for $Q_\me$ and $Q_\mv$.
In this way, $\gamma(u)\in Y$ imposes that the values $u_\me(\mv),u_\mf(\mv)$ agree for any two edges $\me,\mf$ with common tail $\mv$, up to proper weights:
\[\frac{u_\me(\mv)}{\omega_{\mv,\me}} = \frac{u_\mf(\mv)}{\omega_{\mv,\mf}}. \]

Observe that $\dim \Ran B = \rank (\mathcal I^+_\omega ) = |\mV|$, so $B$ is surjective and by \autoref{rem:bcglob}  we have 
\[Z= Y^\perp \oplus \Ran B^\ast {= \Ker (\mathcal I^-_\omega)  \oplus \Ran(\mathcal I^+_\omega)^\top} \subset \C^{2|\mE|}. \]
In this case $\dim  \Ran B^\ast =  \dim \Ran B = |\mV|$ and
$\dim Y^\perp = \dim \Ker (\mathcal I^-_\omega) = |\mE| - |\mV|$ { by the Rank-Nullity-Theorem}.
Accordingly, condition \eqref{eq:basis-global} is satisfied.
Hence,  the system has the right number of transmission conditions and we recover contractive well-posedness by our \autoref{prop:main1dyn}. 
We can easily apply the results in Section~\ref{sec:qualitative} and deduce that the semigroup is real (resp., positive) if and only if the matrix $\mathcal C$ is real (resp.,  has nonnegative off-diagonal entries). Furthermore, if the semigroup generated by $\mathcal C$ (resp., $\mathcal C^*$) on $Y^{(d)}$ is contractive with respect to the $\infty$-norm, then the semigroup generated by $\mathbb A$ is contractive on $L^\infty_d(\mathcal G)$ (resp., on $L^1_d(\mathcal G)$), hence on $L^p_d(\mathcal G)$ for all $p\in [2,\infty]$ (resp., for $p\in [1,2]$).

\subsection{Telegrapher's equations}

The $2\times 2$ hyperbolic system
\begin{equation}\label{syst2time2}
\left\{\begin{array}{ll}
\dot{p}+L q' +G p+Hq=0 \quad\hbox{ in } (0,\ell) \times(0,+\infty),\\
\dot{q}+P p' +K q+Jp=0 \quad\hbox{ in } (0,\ell)\times(0,+\infty),
\end{array}
\right.
\end{equation}
on a real interval $(0,\ell)$ generalizes the first order reduction of the wave equation and offers a general framework to treat models that appear in several applications. The analysis of this system on networks
with different boundary conditions 
has been performed in \cite{Nic:2017}.

In electrical engineering
\cite{MaffucciMiano:06,ImpJol14, Bec16},
$p$ (resp. $q$) represents the voltage $V$ (resp. the electrical current $I$) at $(x, t)$,
 $H=J=0$, $L=\frac{1}{C}$, $P=\frac{1}{L}$,
$G=\frac{\hat{G}}{C}$, $K=\frac{R}{L}$, where $C>0$ is the capacitance, $L>0$ the inductance, $\hat{G}\geq 0$ the conductance,
and $R\geq 0$ the resistance: \eqref{syst2time2} is then referred to as ``telegrapher's equation''. 

Also Maxwell's equations in tube-like 3D domains can be intuitively reduced to a system of 1D networks~\cite{ImpJol12} 
for $P=L=-1$ and $G=H=K=J=0$,  where $p$ (resp. $q$) represents the electric field $E$ (resp. the magnetic field $B$). Accurate asymptotic analysis of the system shows that the 1D model is indeed related to the full 3D model, up to errors that can be estimated~\cite{ImpJol14}; more general settings have been considered in~\cite{BecImpJol14,Bec16}. The 1D Maxwell's equations is also derived from physical principles in~\cite[\S~2]{Whe97}, thus obtaining again a special instance of~\eqref{syst2time2}.

Assuming that $L,P$ are two real numbers both positive or both negative, \autoref{assum-MQN} hold for system
\eqref{syst2time2} with $u_\me=(p_\me, q_\me)^\top$ and
\begin{equation*}\label{eq:Qhypsys}
M_\me=
-
\begin{pmatrix}
0 & L\\ P & 0
\end{pmatrix},\quad
N_\me=-\begin{pmatrix}
G & H\\ K & J
\end{pmatrix},
\quad \hbox{and}\quad
Q_\me=\begin{pmatrix}
|P| & 0\\ 0 & |L|
\end{pmatrix}.
\end{equation*}
In such a case, we see that
\[
Q_\me M_\me=\begin{pmatrix}
0 & L |P|\\ L |P| & 0
\end{pmatrix}.
\]

Since telegrapher's equation \eqref{syst2time2}
on networks
with non dynamic boundary conditions from \cite{Carlson:11,Nic:2017} enters into the framework
of~\cite{KraMugNic20}, we 
here concentrate on dynamic boundary conditions. 
We first start with a simple example and then consider a system set on a star-shaped network.

\subsubsection{Maxwell system with dynamic boundary conditions}\label{exa:maxdynamic}
Let us study the Maxwell system
\begin{equation}\label{eq:max1old}
\left\{
\begin{split}
\dot{p}&=q',\\
\dot{q}&= p',
\end{split}
\right.
\end{equation}
a special case of~\eqref{syst2time2}, on two adjacent intervals $\me_1=(-1,0)$ and $\me_2=(0,1)$ (with common vertex $\mv_0\equiv 0$). We denote by
$u_i:=(p_i, q_i)^\top$ the unknowns on the edge $\me_i$, $i=1,2$. We impose electric boundary condition at $-1$ and the magnetic condition at $1$
complemented by continuity of $p$ in 0 
along with a dynamic boundary condition.
This means that the boundary/dynamic conditions can be written as
\begin{eqnarray}
p_1(t,-1)&=&q_2(t,1)=0,\label{eq:condelectric}\\
p_1(t,0)&=&p_2(t,0),\label{eq:condDyn1}\\
\frac{d}{dt}p_{1}(t,0)&=&q_2(t, 0)- q_1(t, 0).\label{eq:condDyn2}
\end{eqnarray}
To write the system in the formalism introduced in Section~\ref{sec:dynam} we  define 
\begin{eqnarray*}
\gamma_{\mv_{-1}}(u)&:=&(p_1(-1),q_1(-1)) \subset\C^2,\\
\gamma_{\mv_1}(u)&:=&(p_2(1),q_2(1)) \subset\C^2,\\
\gamma_{\mv_0}(u)&:=&(p_1(0),q_1(0),p_2(0),q_2(0)) \subset\C^4.
\end{eqnarray*}
In the vertices $\mv_{-1}$ and $\mv_1$ we only have stationary boundary conditions \eqref {eq:condelectric} which are satisfied by taking
\[ Y_{\mv_{-1}}:= \{0\}\oplus\C\quad\text{and}\quad  Y_{\mv_1}:= \C\oplus\{0\}.\]
In $\mv_0$ we  enforce the stationary  condition \eqref{eq:condDyn1} by taking
 $Y_{\mv_0}:=\{ (1,0,-1,0)^\top\}^\perp$ while for the dynamic condition  we take
\[Y_{\mv_0}^{(d)}:=\lin\{ (1,0,1,0)^\top\} \subset Y_{\mv_0},
\]
and define
\[
B_{\mv_0}:=\left(\begin{array}{llll}
0&-1&0&1\\
0&0&0&0\\
0&-1&0&1\\
0&0&0&0
\end{array}
\right).
\]
With this choice,  we see that  
 \eqref{eq:condDyn2} is equivalent to
 \begin{equation}\label{DynEE}
 \dot{\mmx}_{\mv_0}= B_{\mv_0} \gamma_{\mv_0}(u) \quad\text{where}\quad \mmx_{\mv_0}=P_{\mv_0}^{(d)} \gamma_{\mv_0}(u)\text{ and } \gamma_{\mv_0}(u)\in Y_{\mv_0}.
 \end{equation}
 Now, by taking $Q_{\mv_0}=I$, we notice that the boundary term in \eqref{eq:Adissdyn} corresponding to
 $\mv_0$ is equal to
 \[
 \Re\left(p_1(0)\bar q_1(0)-p_2(0)\bar q_2(0)+(q_2(0)-q_1(0))\bar p_1(0)\right),
 \]
 which by \eqref{eq:condDyn1} is zero. Similarly,  due to the boundary condition at  the two endpoints $\mv_{-1}, \mv_1$, their  corresponding boundary terms in \eqref{eq:Adissdyn} are zero. We are thus in the setting of \autoref{rem-Z and global}.(2).
  
 So, it remains to check {\eqref{eq:basis}}. But as $B_{\mv_0} $ is surjective, $
Z_{\mv_0}$ is given by \eqref{defZmvBsurjective}
and since $\Ran B_{\mv_0}^\ast=\lin\{(0, -1, 0, 1)^\top\}$, we find
\[
\widetilde Z_{\mv_0} = Z_{\mv_0}  = \lin\{ (1,0,-1,0)^\top, (0, -1, 0, 1)^\top \}
\]
For the two endpoints $\mv_{-1}, \mv_1$,
 we only have stationary conditions, hence $Z_{\mv_{-1}} = Y_{\mv_{-1}}^\perp$, $Z_{\mv_{1}} = Y_{\mv_{1}}^\perp$ and
 \[ \widetilde Z_{\mv_{-1}}  = \C\oplus\{0\}\oplus\{0\}\oplus\{0\},\quad  \widetilde Z_{\mv_{1}}  = \{0\}\oplus\{0\}\oplus\{0\}\oplus \C. \]
It is now easy to verify the dimension equation \eqref{eq:basis-surj}, hence \autoref{cor:group} can be applied.  We finally obtain that the considered problem is governed by a unitary group.

According to \autoref{prop:real-dyn}  the  group is real since all involved constants are real, but we may expect that it   does not preserve positivity and is not $\infty$-contractive  since $M_\me$ is not diagonal.

\subsubsection{Telegrapher's equations with dynamic boundary conditions}\label{exa:Imperialedynamic}
Here we analyze the electrical formulation of system \eqref{syst2time2} on a $Y$-shaped structure
 with the transmission conditions from \cite[\S 8.2]{Bec16}
at the common vertex (called the improved Kirchhoff condition). Hence, in reference to the electrical interpretation, we assume that
 $P$ and $L$ are two positive constants, $H=J=G=K=0$, further
$p$ (resp. $q$) is denoted by $V$ (resp. $I$).
More precisely, the network consists of three edges $\me_i$, $i=0,1,2$
identified with $(0,1)$ having a common vertex $\mv_1\equiv 0$, where the edge $\me_0$ plays a specific rule
since the transmission conditions at $0$ from \cite[(8.9)]{Bec16} are given by
\begin{equation}\label{eq:boundary2intro}\begin{split} 
\sum_{k=1}^{2}{\cL}_{j k} \dot{I}_{k}(t,0)&=V_{0}(t,0)-V_{j}(t,0) \quad \text{for }j \in \{1,2\}, t>0,\\
\dot{V}_{0}(t,0)&=-\sum_{j=0}^{2}I_{j}(t,0)\quad t>0
\end{split}
\end{equation}
 where 
 $\cL=(\cL_{j k})_{2 \times 2}$ is a symmetric, real-valued positive definite matrix.
 Here, for simplicity we take all the other coefficients equal to 1. 
 At the endpoints, we take the boundary conditions 
\begin{equation}
\label{eq:electricimperial}
I_0(t,1)=V_j(t,1)=0 \quad \text{ for } j=1,2.
\end{equation}

To write the system in our formalism, we define
\[
\gamma_{\mv_1}(u)=( I_1(0), I_2(0), I_0(0), V_1(0), V_2(0), V_0(0))^\top,
\] 
so that $\C^{k_{\mv_1}}=\C^6$. 
Since only dynamical conditions are imposed at $\mv_1$, we take 
 $Y_{\mv_1}:=\C^6$, we choose
 \[Y_{\mv_1}^{(d)}:=\lin\{ (1,0,0,0,0,0)^\top, (0,1,0,0,0,0)^\top, (0,0,0,0,0,1)^\top\},
\] and we define  { 
\[B_{\mv_1}:= 
 \left(
 \begin{array}{cccccc}
 0& 0& 0& -a_{11}& -a_{12}& a_{11}+a_{12}\\
 0& 0& 0&-a_{12}& -a_{12}& a_{12}+a_{22}\\
 0&0&0&0&0&0\\
0&0&0&0&0&0\\
0&0&0&0&0&0\\
- 1& -1& -1& 0&0&0
 \end{array}
 \right),
 \]
where 
$\cL^{-1}=\left(
 \begin{array}{cc}
 a_{11}& a_{12}  \\
 a_{12}& a_{22} 
 \end{array}
 \right)$.}
 With these notations we see that \eqref{eq:boundary2intro} is equivalent to 
 \[\dot{\mmx}_{\mv_1}= B_{\mv_1} \gamma_{\mv_1}(u)\quad
\text{ where }\quad\mmx_{\mv_1}=P_{\mv_1}^{(d)} \gamma_{\mv_1}(u).\]
 Now, by taking 
 $Q_{\mv_1}=PL \diag(\cL, 1)$, we notice that the boundary term in \eqref{eq:Adissdyn} corresponding to
 $\mv_1$ is equal to 0.
 
 We immediately check that $\Ran B_{\mv_1}= Y_{\mv_1}^{(d)}$
hence, by \eqref{defZmvBsurjective}, $Z_{\mv_1} = \Ran B_{\mv_1}^{\ast}$.
Further, 
 \eqref{eq:electricimperial} 
 yields
\[
\tilde \mw^{(\mv_2, 1)}=(0,0,1,0,0,0)^\top, \tilde\mw^{(\mv_3, 1)}=(0,0,0,1,0,0)^\top,
\tilde\mw^{(\mv_3, 2)}=(0,0,0,0,1,0)^\top.
\]
Since the three columns of $B_{\mv_1}^{\ast}$ and 
 these three vectors form a basis of $\mathbb{C}^6$, \autoref{cor:group} shows that the considered problem is governed by a group of isometries.
 
 Note if in \eqref{syst2time2} we allow 
 $H, J, G$ and $K$ to be different from zero, the considered problem is governed by a group.
 
 As before, according to \autoref{prop:real-dyn} , the  group is real since all involved constants are real, but  we are not able to say anything about positivity or $\infty$-contractivity since $M_\me$ is not diagonal.

\subsection{Second sound in networks}\label{sec:secsound}
A wave-like form of thermal propagation has been conjectured to exist in ultracold gases by Lev Landau and is now known under the name of ``second sound''; it has ever since been experimentally observed in several molecules. One classical model boils down to the linear equations of thermoelasticity 
 \begin{equation}\label{systsecondsound}
\left\{
\begin{array}{rcll}
\ddot{z}-\alpha z'' + \beta \theta'&=&0 \quad &\hbox{ in } (0,\ell) \times(0,+\infty),\\
\dot{\theta}+\gamma q'+\delta \dot{z}'&=&0 \quad &\hbox{ in } (0,\ell) \times(0,+\infty),\\
\tau_0 \dot{q}+q + \kappa \theta'&=&0 \quad&\hbox{ in } (0,\ell) \times(0,+\infty),
\end{array}
\right.
\end{equation}
where $z$, $\theta$, and $q$ represent the displacement, the temperature difference to a fixed reference temperature, and the heat flux, respectively, and $\alpha, \beta,\gamma, \delta, \tau_0,\kappa$ are positive constants. Racke has discussed in~\cite{Rac02} the asymptotic stability of this system under three different classes of boundary conditions, including
\[
\alpha z'(0)=\beta\theta(0),\quad \theta'(0)=0,\quad z(\ell )=\theta(\ell)=0.
\]
While he does not point it out explicitly, this leads indeed to a dynamic condition: indeed, $\theta'(0)$ is not well-defined if $\theta$ is merely of class $H^1$, but assuming that the initial data are smooth enough that the third equation in~\eqref{systsecondsound} can be evaluated at 0,
yielding
\[
\tau_0 \dot{q}(0)+q(0) + \kappa \theta'(0)=0
\]
the condition $\theta'(0)=0$ leads to
\begin{equation}\label{Dynbcsecondsound}
 \dot{q}(0)=-\frac{1}{\tau_0} q(0),
\end{equation}
which can indeed be made sense of even for general initial data, and then studied by the method introduced in the previous section. In summary we now study system \eqref{systsecondsound}
with the dynamic boundary condition \eqref{Dynbcsecondsound} and the stationary ones
\begin{equation}\label{bcsecondsound}
\alpha z'(0)=\beta\theta(0),\quad \quad z(\ell )=\theta(\ell)=0.
\end{equation}

We observe that \autoref{assum-MQN} are satisfied taking $u = (z', \dot{z}, \theta, q)$,
\[
M_\me := \begin{pmatrix}
 0& 1& 0& 0\\
 \alpha& 0& -\beta&0\\
 0& -\delta& 0 &-\gamma\\
 0& 0& -\frac{\kappa}{\tau_0}& 0
\end{pmatrix},\quad
Q_\me := \begin{pmatrix}
 \alpha\delta& 0& 0& 0\\
 0& \delta& 0&0\\
 0& 0 &\beta &0\\
 0& 0& 0 &\frac{\beta\gamma\tau_0}{\kappa}
\end{pmatrix},\quad \text{and}\quad
N_\me := \begin{pmatrix}
 0 & 0& 0& 0\\
 0& 0 & 0&0\\
 0& 0 & 0 &0\\
 0& 0& 0 & -\frac{1}{\tau_0}
\end{pmatrix}.\]
A direct computation shows that
\[
Q_\me M_\me=\begin{pmatrix}
0 & \alpha\delta & 0 & 0\\
\alpha\delta & 0 & -\beta\delta & 0\\
0 & -\beta\delta & 0 & -\beta \gamma\\
0 & 0 & -\beta\gamma & 0
\end{pmatrix}
\]
with four eigenvalues of the form
$\pm\sqrt{\frac{H\pm2\sqrt{K}}{2}},$
where $H:= \alpha^2 \delta^2+\beta^2 \delta^2+ \beta^2\gamma^2$ and
$K:= H^2 - 4 \alpha^2\beta^2\gamma^2\delta^2$.
 Because $H^2>K$ whenever $\alpha,\beta,\gamma,\delta>0$, $Q_\me M_\me$ has two positive and two negative eigenvalues. This is consistent with the above choice \eqref{Dynbcsecondsound}-\eqref{bcsecondsound}
of boundary conditions in the purely hyperbolic case of $\tau_0>0$.

If the endpoint $0$ (resp. $\ell$) is identified with $\mv_1$ (resp. $\mv_2$), we take
\[
Y_{\mv_1}= \left\{x\in \C^4: x_1=\frac{\beta}{\alpha} x_3\right\},\quad 
Y^{(d)}_{\mv_1}= \{0\}\oplus \{0\}\oplus \{0\}\oplus \C\subset Y_{\mv_1},
\]
and
\[
Y_{\mv_2}=\C\oplus \{0\}\oplus \{0\}\oplus \C.
\]
Observe that $\gamma_{\mv_1}(u)=u(0)\in Y_{\mv_1}$, $\gamma_{\mv_2}(u)=u(\ell)\in Y_{\mv_2}$ return all stationary conditions, whereas
\[
\frac{d}{dt}P^{(d)}_{\mv_1}\gamma_{\mv_1}(u)=C_{\mv_1} P^{(d)}_{\mv_1}\gamma_{\mv_1}(u)
\]
with 
\[
C_{\mv_1}=\begin{pmatrix}
0 & 0 & 0 & 0\\
0 & 0 & 0 & 0\\
0 & 0 & 0 & 0\\
0 & 0 & 0 & -\frac{1}{\tau_0}
\end{pmatrix}
\]
corresponds to the  dynamic condition~\eqref{Dynbcsecondsound}. Also, observe that 
$B_{\mv_1}=0$, therefore by \eqref{defZmv} we have
\[
Z_{\mv_1}= Y_{\mv_1}^{\perp} + \Ker B_\mv^\ast=Y_{\mv_1}^{\perp} +  Y^{(d)}_{\mv_1}=\lin\{ (\alpha ,0,-\beta,0)^\top, (0, 0, 0, 1)^\top\}.
\]
Furthermore,
 \[
{Z_{\mv_2}=} Y_{\mv_2}^\perp= \lin\{ (0 ,1,0,0)^\top, (0 ,0,1,0)^\top\},
\]
and it is easy to see that \eqref{eq:basis} applies.

 Now, by taking 
 $Q_{\mv_1}=\tau_0 \beta$, we notice that the boundary term in~\eqref{eq:Adissdyn-2}  corresponding to
 $\mv_1$ is equal to 
 \begin{eqnarray*}
 2\beta \gamma \Re (\theta(0)\overline q(0))-\beta |\theta(0)|^2
 &=&-|\gamma q(0)-\beta \theta(0)|^2+ \gamma^2 |q(0)|^2
 \\
 &\leq& \gamma^2 |q(0)|^2= \gamma^2 |P_{\mv_1}^{(d)} \gamma_{\mv_1}(u)|^2.
 \end{eqnarray*}
 Hence,  in view of \autoref{prop:main1dyn}, the system is well-posed. More precisely, the initial value problem associated with~\eqref{systsecondsound} with the above boundary conditions is governed by a strongly continuous semigroup on $L^2_d(\mathcal G)\equiv L^2(0,\ell)\oplus Y^{(d)}$.

As before, according to \autoref{prop:real-dyn},  the  semigroup is real since all involved constants are real, but again  positivity and  $\infty$-{contractivity cannot  be checked by our abstract results}  since $M_\me$ is not diagonal.

 System \eqref{systsecondsound} on a network with  stationary boundary conditions at the nodes, namely  continuity of $z$ and $q$ and Kirchhoff-type conditions for $z'$ and $\theta$,
 were described in~\cite[\S~5.6]{KraMugNic20}. With the method described above, we can e.g.\ impose dynamic conditions on the vertex evaluation of $z$ and/or $q$ at an arbitrary subset of $\mV$, while keeping Kirchhoff-type conditions  for $z'$ and $\theta$, still retaining a well-posed system.

\subsection{Wave type equations\label{exa:waves}}

Wave-type equations on graphs   have retained the attention of many authors, {see  \cite{Ali94,LagLeuSch94,Mug14,AMbook} and the references cited there. 
Here we 
show that our framework can be applied} to rather general elastic systems modeled as 
\begin{equation}\label{eq:waves}
\ddot{u}_\me(t,x)=u_\me''(t,x)+\alpha_\me \dot{u}_\me'(t,x)+\beta_\me \dot{u}_\me(t,x)+\gamma_\me u'_\me(t,x),\qquad t\ge 0,\ x\in (0,\ell_\me),
\end{equation}
where $\alpha_\me \in C^1([0,\ell_\me])$
and $\beta_\me, \gamma_\me\in L^\infty (0,\ell_\me)$ are real-valued functions. For the sake of simplicity, as in~\cite[\S5.23]{KraMugNic20} we restrict ourselves to stars with $J\geq 2$ edges as in Figure~\ref{fig:star}, which can be regarded as building blocks of more general networks, but contrary to ~\cite[\S5.23]{KraMugNic20} we assume 
 that the edges are connected by a point mass at their common vertex, see
 \cite{Castro:97,HansenZuazua:95}
 for $J=2$ and the wave equation, i.e., $\alpha_\me =\beta_\me=\gamma_\me =0$
 (see also \cite{MorgulRaoConrad:94} for a cable with a tip mass).
\begin{figure}[H]
\begin{center}
\begin{tikzpicture}
\foreach \x in {25,45,-25,-45}{
\draw[-{Latex[scale=1.5]}] (\x:0cm) -- (\x:3.6cm);
\draw[fill] (\x:3.6cm) circle (2pt);
}
\draw[-{Latex[scale=1.5]}] (180:3.6cm) -- (180:0cm);
\draw[fill] (180:3.6cm) circle (2pt);
\draw[fill] (0:0cm) circle (2pt) node[above]{$\mv_0$};
\node at (180:3.6) [anchor=south] {$\mv_1$};
\node at (180:1.8) [anchor=south] {$\me_1$};
\node at (45:3.6) [anchor=west] {$\mv_2$};
\node at (45:1.8) [anchor=south] {$\me_2$};
\node at (25:3.6) [anchor=west] {$\mv_3$};
\node at (25:1.8) [anchor=south] {$\me_3$};
\node at (-25:3.6) [anchor=west] {$\mv_{J-1}$};
\node at (-25:1.8) [anchor=south] {$\me_{J-1}$};
\node at (-45:3.6) [anchor=west] {$\mv_J$};
\node at (-45:1.8) [anchor=south] {$\me_J$};
\foreach \y in {20,15,10,5,0,-5,-10,-15,-20}{
\draw[fill] (\y:2.4cm) circle (.5pt);
}
\end{tikzpicture}
\caption{A star-shaped network with one incoming and $J-1$ outgoing edges.}\label{fig:star}
\end{center}
\end{figure}
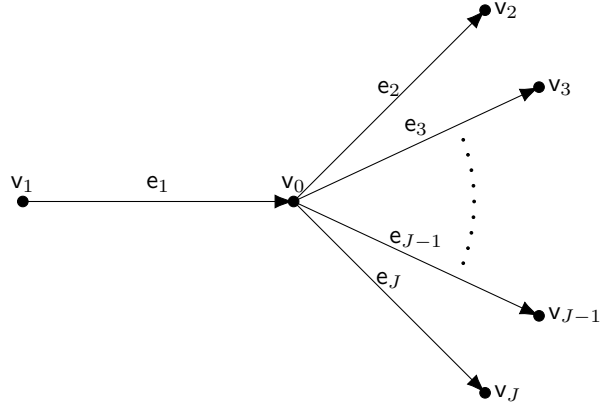
It turns out that \eqref{eq:waves} is equivalent to
\[
\dot{U}_\me= M_\me U_\me'+ N_\me U_\me,
\]
for the vector function $U_\me=(u_\me', \dot{u}_\me)^\top
$, where
\[
M_\me=\begin{pmatrix}
0 & 1 \\ 1 & \alpha_\me
\end{pmatrix},\qquad
N_\me=\begin{pmatrix}
 0 & 0\\
 \gamma_\me & \beta_\me \\
\end{pmatrix}.
\]
As $M_\me$ is symmetric, \autoref{assum-MQN} are automatically satisfied by choosing $Q_\me$ as the identity matrix.
As before, the boundary conditions at the vertices are related to the values of $M_\me$
at
the endpoints of the edge $\me$, that generically are given by
\[
M_\me(\mv)=\begin{pmatrix}
0& 1 \\ 1 & \alpha_\me(\mv)
\end{pmatrix},
\]
when $\mv$ is one of the endpoints of $\me$; hence $M_\me(\mv)$ has two real eigenvalues of opposite sign, 
\[
\lambda_\pm=\frac{1}{2}\left(\alpha_\me(\mv)\pm\sqrt{\alpha_\me(\mv)^2+4}\right).
\]
We then need $J$ boundary conditions at the common node $\mv_0$
and one boundary condition at each endpoint $\mv_i$, $i=1,\dots, J$.

For an exterior vertex $\mv_i$ ($i=1,\dots, J$)  we choose 
 Dirichlet boundary condition
 \[
u_{\me_i}(\mv_i)=0,
 \]
 that leads to
 $\dot{u}_{\me_i}(\mv_i)=0,$
and corresponds to the choice of $Y_{\mv_i}$ spanned by {$(1,0)^\top$} that is a totally isotropic subspace associated with $T_{\mv_i}$, {whereby $T_{\mv_1} = -M_{\me_1}(\mv_1)$ and $T_{\mv_i} = M_{\me_i}(\mv_i)$ for $i=2,\dots,J$.} 
We refer to~\cite[\S5.2]{KraMugNic20} for other boundary conditions at the exterior vertices.

 Now, inspired by \cite{Castro:97,HansenZuazua:95}, we impose the following boundary conditions at $\mv_0$, namely continuity of $u_\me$ at $\mv_0$ and 
\begin{equation}\label{eq:boundarymass} 
 {
-\sum_{i=1}^J u_{\me_i}'(\mv_0) \iota_{\mv_0{\me_i}}= \delta \ddot{u}_{\me_1}(\mv_0),
}
\end{equation}
for some positive constant $\delta$. Let us check that such a boundary condition corresponds to a dynamical one.
Indeed,  the continuity condition of $u_\me$ at $\mv_0$
implies that
{\[
\gamma_{\mv_0}(U)\in Y _{\mv_0} :=
\{(x,y)^\top\ :\ x\in \C^J, y=\alpha{\mathbf 1}, \alpha \in \C \}=\left(\C^{J}\oplus \{{\mathbf 0^\top}\}\right)\oplus
\lin \{({\mathbf 0}, {\mathbf 1})^\top\},
\]}
where we write
\[
\gamma_{\mv_0}(U) :=(u'_{\me_1}(\mv_0), \dots, u'_{\me_J}(\mv_0), \dot{u}_{\me_1}(\mv_0), \dots, \dot{u}_{\me_J}(\mv_0))^\top,
\]
and
 {${\mathbf 1}, {\mathbf 0}$ are the row vectors in $\C^{J}$} whose  all entries equal  1 and $0$, respectively. In order  to formulate \eqref{eq:boundarymass} in our setting, we
 set
 \[
Y^{(d)} _{\mv_0} := \lin \{ {({\mathbf 0}, {\mathbf 1})}^\top\}\subset Y _{\mv_0},
 \]
 and introduce $B_{\mv_0}$ as the  {$2J\times  2J$} matrix  
 \[ {
{B}_{\mv_0} := -\frac{1}{\delta}
 \left(
 \begin{array}{llll}
 {\mathbf 0}&{\mathbf 0}\\
 \vdots&\vdots
 \\
 {\mathbf 0}& {\mathbf 0}\\
 \iota_{\mv_0,\ast}& {\mathbf 0}\\
  \vdots&\vdots
 \\
\iota_{\mv_0,\ast}& {\mathbf 0}\\
 \end{array}
 \right), }
 \]
{where $\iota_{\mv_0,\ast}$ is the row of the incidence matrix $\mathcal I$ corresponding to $\mv_0$.}
 We then readily see that \eqref{eq:boundarymass} is equivalent to
{
 \[
 \dot\mmx_{\mv_0}= B_{\mv_0} \gamma_{\mv_0}(U),
 \]}
 where 
 ${\mmx_{\mv_0}}=P^{(d)}_{\mv_0} \gamma_{\mv_0}(U)$, recalling the continuity condition of $\dot{u}_\me$ at $\mv_0$.

 Now, by taking 
 $Q_{\mv_0}=\delta$, we notice that the boundary term in \eqref{eq:Adissdyn} corresponding to
 $\mv_0$ is equal to 
 \begin{eqnarray*}
\frac{|\dot{u}_\me(\mv_0)|^2}{2}\sum_{i=1}^J\alpha_{\me_i} (\mv_0) \iota_{\mv\me_i}
=
\frac{|P^{(d)}_{\mv_0} \gamma_{\mv_0}(U)|^2}{2}\sum_{i=1}^J \alpha_{\me_i} (\mv_0)\iota_{\mv\me_i}.
 \end{eqnarray*}

Finally, we readily check that $\Ran B_{\mv_0}= Y_{\mv_0}^{(d)}$, hence by  \eqref{defZmvBsurjective} and since $\Ran B_{\mv_0}^\ast=Y_{\mv_0}^{(d)}$ as well, we find
\[
Z_{\mv_0}=  {\{{\mathbf 0}^\top \} \oplus \C^J.}
\]
Further, as {$Z_{\mv_i} = Y_{\mv_i}^\perp$} and $\sum_{i=1}^J \widetilde Y_{\mv_i}^\perp={\C^{J}\oplus \{{\mathbf 0}^\top\}}$, 
{ \eqref{eq:basis-surj} holds for $k=2J$} and we conclude that system \eqref{eq:waves}
with the previous boundary conditions is governed by a group.

 In conclusion owing to \autoref{prop:main1dyn} the system is well-posed. More precisely, the initial value problem associated with~\eqref{systsecondsound} with the above boundary conditions is governed by a strongly continuous  group on $L^2_d(\mathcal G)\equiv L^2(0,\ell)\oplus Y^{(d)}$.

As before according to \autoref{prop:real-dyn}  the  semigroup is real since all involved constants are real, but again assessing either positivity or  $\infty$-contractive is problematic since $M_\me$ is not diagonal.

We have discussed in~\cite{KraMugNic20} how our formalism can be used to study networks of beams under rather general transmission conditions of stationary type. We restrain from elaborating on this topic, but it should by now be clear to the reader that suitable, \textit{different} choices of $Y_\mv$ (cf.\ \autoref{rem:comp-par-hyp}), and of course suitable choices of $Y^{(d)}_\mv$, promptly lead to models of networks of beams with dynamic transmission conditions, which can then be studied by our theory. We mention that comparable well-posedness results have been recently obtained in~\cite{GreMug20b}.

{
\subsection{The Dirac equation}\label{sec:dirac}
The 1D Dirac equation on a network, as studied in~\cite{BolHar03}, takes on each edge the form
\[
\imath \hbar \frac{\partial}{\partial t} \psi=\left(\hbar c\begin{pmatrix}0 & -1\\ 1 & 0\end{pmatrix} \frac{\partial }{\partial x} +mc^2 \begin{pmatrix}1 & 0\\ 0 & -1\end{pmatrix}\right)\psi\
\]
for a $\C^2$-valued unknown $\psi=(\psi^{(1)},\psi^{(2)})$. A parametrization of skew-adjoint realizations on a network has been presented in~\cite{BolHar03} and in~\cite{KraMugNic20} we have taken advantage of our theory and provided further realizations generating (semi)groups, since our \autoref{assum-MQN} are satisfied letting
\[
M_\me =\begin{pmatrix}
0 & \imath c\\
-\imath c & 0
\end{pmatrix},\quad
Q_\me =\begin{pmatrix}
1 & 0\\
0 & 1
\end{pmatrix},\quad\hbox{and}\quad N_\me =\begin{pmatrix}
- \imath \frac{mc^2}{\hbar} & 0\\
 0 & \imath \frac{mc^2}{\hbar}
\end{pmatrix},\qquad \me\in\mE.
\]
{
Let us now study the quadratic form {$q_\mv$}, cf.~\eqref{eq:qv-def}.
We first  observe 
that {$T_\mv$ is a $2|\mE_\mv|\times 2|\mE_\mv|$ block diagonal matrix with diagonal blocks equaling $M_\me \iota_{\mv\me}$.
Hence, 
if we write
\[
\gamma_\mv(U):=(\psi^{(1)}_{\me}(\mv), \psi^{(2)}_{\me}(\mv))_{\me\in \mE_{\mv}},
\]
then
 $(\xi,\eta)^\top\in \C^{2|\mE_\mv|}$, with $\xi := (\psi^{(1)}_\me(\mv))_{\me\in\mE_\mv}$ and $\eta :=(\psi^{(2)}_\me(\mv))_{\me\in\mE_\mv}$} 
 is an isotropic vector for the associated quadratic form {$q_\mv$} if and only if 
\begin{equation}\label{eq:vanishingima-new}
\sum_{\me\in \mE_\mv} \iota_{\mv \me} \Im(\xi_{\me} \bar \eta_{\me})=0.
\end{equation}
A somewhat canonical choice is that of conditions of continuity and of Kirchhoff-type    on $\psi^{(1)}$ and $\psi^{(2)}$, respectively, at each $\mv\in\mV$; this fits in our abstract framework by letting 
 \[Y_\mv:=\lin\{{\mathbf 1}_{\mE_\mv}\} \oplus \lin\left\{\iota_{\mE_\mv}\right\}^\perp
 \]
  { (we recall that $\mathbf{\iota}_{\mE_\mv}$ denotes the vector in $\C^{|\mE_\mv|}$ whose $\me$-th entry is $\iota_{\mv\me}$)}
  and is easily seen to lead to a hyperbolic system governed by a unitary group. Further instances of the Dirac equation governed by a unitary group, and hence with a quantum mechanical significance, can be easily produced applying the theory presented above: we will only focus on one such realization. By keeping the continuity property of $\psi^{(1)}$ at the vertices, we here take  
\begin{equation}
 \label{defYvDirac-new-b}
Y_\mv:=\lin\{{\mathbf 1}_{\mE_\mv}\}\oplus \C^{{|\mE_\mv|}},
\end{equation}
and  we  let 
 \[Y^{(d)}_\mv:=\lin\{{\mathbf 1}_{\mE_\mv}\}\oplus \{\mathbf 0_{\mE_\mv}\}\subset  Y_\mv.
 \]
Let us finally define
 \[
 B_\mv \colon Y_\mv\ni  (\xi,\eta)^\top\mapsto -\imath (\eta\cdot \iota_{\mE_\mv}) ({\mathbf 1}_{\mE_\mv}\oplus {\mathbf 0}_{\mE_\mv}) \in  Y_\mv^{(d)}
 \]
  and
 \[
 {C_\mv}: Y^{(d)}\ni (\xi,0)^\top\mapsto ( {C_\mv}^{(1)}\xi,0)^\top\in  Y^{(d)},
 \] 
 for any { skew-hermitian $|\mE_\mv|\times |\mE_\mv|$-matrix $C_\mv^{(1)}$.}
 This corresponds to imposing
 \begin{itemize}
 \item continuity conditions across each vertex on $\psi^{(1)}$ as well as
 \item dynamic conditions
 \[
 \frac{d\psi^{(1)}}{dt}(t,\mv)=
 { - \imath\sum_{\me\in\mE_{\mv}}  \psi^{(2)}_{\me}(t,\mv) \iota_{\mv\me}+C_\mv^{(1)}\psi^{(1)}(t,\mv),}\qquad \mv\in\mV.
 \]
\end{itemize} 

{
Observe that
 \begin{equation}\label{eq:dimens-dirac-0}
 \dim Y_\mv = 1+ {|\mE_\mv|}, 
 \qquad \dim Y^{(d)}_\mv = 1,
 \end{equation}
 {but this is not sufficient to guarantee \eqref{eq:basis-surj}  and thus \eqref{eq:basis}.}
{As $B_\mv$ is  surjective, simple calculations show that (using the parametrization of the edges so that for both of them, $\mv_1$ is identified with 0 and 
$\mv_2$ is identified with 1)
\[
Z_{\mv_1}=Z_{\mv_2}=\lin \{(1,0,-1,0)^\top, (0,1,0,1)^\top\}, 
\]
hence \eqref{eq:basis} cannot hold.} 

{However, by taking 
 $Q_{\mv}:=c \mathbb{I}$ at each $\mv\in\mV$ one can show that $\mathbb A^\ast=-\mathbb A$.
 Furthermore, the  boundary terms in~\eqref{eq:Adissdyn-2}  vanish.
 Indeed this is clear by assumptions for the term involving $C^{(1)}$, whereas the latter boundary term is equal to
\[
-c\sum_{\me\in \mE_{\mv}}\iota_{\mv\me} \Im(\xi_{\me} \bar \eta_{\me})
-\imath c \Re((\eta\cdot \iota_{\mE_\mv}) (\mathbf{1}_{\mE_\mv}\cdot \xi))=0\qquad\hbox{for all }\mv\in\mV,
\]
since $\xi_\me=\xi,$ for all $ (\xi,\eta)^\top\in Y_\mv$, and $\Im(z)=-\Re( \imath z)$, for all $z\in \C$.
Hence, we can} invoke \autoref{cor:group-2} and \autoref{rem-Z and global-2} with $\alpha=\beta=0$ and deduce that  $\mathbb A$ generates a  unitary  group on $\bL^2_d(\mathcal G)$. This is a new unitary realization of the Dirac equation that does not appear in the classification in~\cite{BolHar03}, as the latter restricts to stationary vertex conditions.

By \autoref{prop:real-dyn}, this semigroup is not real, hence not positive, either. On the other hand, \autoref{prop:infty-contr-dyn} does not apply, although -- as mentioned in \autoref{rem:noinfty} -- it looks rather plausible that no realization of the Dirac  equation is governed by an $\infty$-contractive semigroup.
}

\bibliographystyle{abbrv}

\end{document}